\documentclass[draft]{article}

\usepackage{amsmath}
\usepackage{amsthm}
\usepackage{amsopn}
\usepackage{amsfonts}
\usepackage{amssymb}
\usepackage{cite}
\usepackage{enumitem}
\usepackage{array,multirow}
\usepackage{float}
\usepackage{xcolor}
\usepackage[margin=1in]{geometry}

\usepackage{tikz}

\theoremstyle{theorem}

\newtheorem{lemma}{Lemma}
\newtheorem{corollary}{Corollary}
\newtheorem{proposition}{Proposition}
\newtheorem{theorem}{Theorem}
\newtheorem*{remark}{Remark}

\DeclareMathOperator{\mypslop}{PSL}
\newcommand{\mypsl}[2]{\mypslop\left(#1, #2\right)}
\DeclareMathOperator{\mypdeltalop}{P \Delta L}
\newcommand{\mypdeltal}[2]{\mypdeltalop\left(#1, #2\right)}
\newcommand{\Tau}{\mathrm{T}}
\DeclareMathOperator{\mybigoop}{\mathcal{O}}
\newcommand{\mybigo}[1]{\mybigoop\left(#1\right)}
\DeclareMathOperator{\mymooreop}{M}
\newcommand{\mymoore}[2]{\mymooreop\left(#1, #2\right)}
\DeclareMathOperator{\mydistop}{d}
\newcommand{\mydist}[2]{\mydistop\left(#1, #2\right)}
\DeclareMathOperator{\mydiamop}{Diam}
\newcommand{\mydiam}[1]{\mydiamop\left(#1\right)}
\DeclareMathOperator{\mydegop}{Deg}
\newcommand{\mydeg}[1]{\mydegop\left(#1\right)}
\DeclareMathOperator{\myautop}{Aut}
\newcommand{\myaut}[1]{\myautop\left(#1\right)}
\DeclareMathOperator{\mydgop}{DG}
\newcommand{\mydg}[2]{\mydgop\left(#1, #2\right)}
\DeclareMathOperator{\mywgop}{WG}

\newcommand{\mywgp}[1]{\mywgop\left(#1\right)}
\DeclareMathOperator{\myvertop}{V}
\newcommand{\myvert}[1]{\myvertop\left(#1\right)}
\DeclareMathOperator{\mycayleyop}{\Gamma}
\newcommand{\mycayley}[2]{\mycayleyop\left(#1, #2\right)}

\tikzstyle{every node}=[
	circle,
	draw,
	black,
	fill=black,
	inner sep=0pt,
	minimum size=2pt,
]
\tikzstyle{every path}=[
	thick,
]

\begin{document}
\title{On Word and G\'omez Graphs and Their Automorphism Groups in the
	Degree-Diameter Problem}
\author{J. Fraser}
\maketitle

\begin{abstract}
One of the prominent areas of research in graph theory is the degree-diameter
problem, in which we seek to determine how many vertices a graph may have when
constrained to a given degree and diameter. Different variants of this problem
are obtained by considering further restrictions on the graph, such as whether
it is directed, undirected, mixed, vertex-transitive etc. The currently best
known extremal constructions are the G\'omez graphs, both when considered as
directed and undirected. As the G\'omez graphs are similar to the
Faber-Moore-Chen graphs, we give a natural generalisation of their definition
which we call \textit{word graphs}. We then prove that the G\'omez graphs are
as large as possible word graphs for given degree and diameter. Further, we provide a test
to determine the full automorphism group of a word graph, and apply it to the
extremal directed G\'omez graphs to settle the previously open question of when
the G\'omez graphs are also Cayley graphs. Finally, we conclude with a brief
list of interesting related open problems.
\end{abstract}

\section{Introduction}

This paper aims to address two principle questions in the degree diameter
problem (a survey of which can be found in \cite{degree-diameter-survey}). The
first question, posed in \cite{when-cayley}, concerns when the G\'omez graphs
are Cayley. To this question we provide a solution in the extremal case, though
note that the paper of G\'omez \cite{gomez} defines a broader family of graphs
than we address here. The second question we address regards the similarity in
the construction methods used in the Faber-Moore-Chen graphs
\cite{faber-moore-chen, faber-moore-chen-2}
and the G\'omez graphs, and shows that with reasonable
assumptions the G\'omez graphs provide an optimal construction.

This paper is divided into material serving four logical purposes.
The first section on tau and sigma sequences provides a technical proof
necessary in subsequent work whose inclusion at a later stage would interrupt
the natural flow of argument. The second section on word graphs contains a
discussion of a natural generalisation of Faber-Moore-Chen graphs and G\'{o}mez
graphs. This section achieves two main goals, firstly an optimality result for
the G\'{o}mez graphs, and secondly providing the motivation to the argument
persued in the remaining section. Sections 4 through 7 deal with the problem of
classifying the full automorphism group of extremal directed G\'omez graphs, and
then showing the limitation of the technique used for dealing with all G\'omez
graphs. The final section concludes with a list of relevant questions which
remain open.

\section{Tau and Sigma Sequences}

In this section we define two special sequences, and we aim to count how many
times a given initial value may occur in these sequences. We define a
\textit{$\tau$-sequence} as an ordered sequence of $n$ integers
$a_1 a_2 \dots a_n$ such that
\begin{enumerate}[label=\textnormal{(\roman*)}]
	\item \label{tau-at-least-one}
		$a_i \geq 0$,
	\item \label{tau-descending}
		if $a_i > 1$ then $a_{i - 1} = a_i - 1$,
	\item \label{tau-is-three-seqs}
		there are at most three $i$ such that $a_i = 0$,
	\item \label{tau-rotate}
		if $a_1 a_2 \dots a_n$ is a $\tau$-sequence then
		$a_n a_1 a_2 a_3 \dots a_{n - 1}$ must also be a
		$\tau$-sequence.
\end{enumerate}
We shall use $\tau(n, \alpha, \beta)$ to indicate the number of
$\tau$-sequences of length $n$ such that $a_1 = \alpha$ and $a_n = \beta$, and
$\tau(n, \alpha)$ the number of $\tau$-sequences of length $n$ such that
$a_1 = \alpha$.

Informally we may think of a $\tau$-sequence of length $n$ as being a
concatenation of either 1, 2 or 3 sequences of ascending integers starting at
0, or a rotation thereof. We now aim to evaluate $\tau(n, i)$ for each
$0 \leq i < n$.

\begin{lemma}
For $n > 1$, $\tau(n, 0, 0) = n - 1$.
\end{lemma}

\begin{proof}
We begin with the observation that in any $\tau$-sequence, if $a_k = 0$ and
$a_{k + j} \neq 0$ for some range of $j$, then we may repeatedly apply
\ref{tau-at-least-one} and \ref{tau-descending} to show that
$a_{k + j} = j$. In particular, each $a_{k + j}$ is uniquely defined.

Now suppose that $a_1 a_2 \dots a_n$ is a $\tau$-sequence with $a_1 = a_n = 0$.
First suppose that there is no $j$ such that $1 < j < n$ and $a_j = 0$. If this
is the case then each $a_j = j - 1$, and there is exactly one
$\tau$-sequence with this property.

Now suppose that there is some $k$ such that $1 < k < n$ and $a_k = 0$. By
\ref{tau-is-three-seqs} there is no $j \not\in \{ 1, k, n \}$ such that
$a_j = 0$. Hence for $1 < j < k$ we must have $a_j = j - 1$ and for
$k < j < n$ we must have $a_{k + j} = j$. Hence there is exactly one
$\tau$-sequence with $a_k = 1$ for each possible value of $k$. This gives rise
to $n - 2$ possible $\tau$-sequences.

Hence, in total, we have $n - 1$ possible $\tau$-sequences of length $n$ with
$a_1 = a_n = 0$, hence $\tau(n, 0, 0) = n - 1$.
\end{proof}

\begin{lemma}
For $1 < i \leq n$, we have $\tau(n, 0, n - i) = i - 1$.
\end{lemma}

\begin{proof}
Suppose that $a_1 a_2 \dots a_n$ is a $\tau$-sequence with $a_1 = 0$. If
$a_n = \alpha > 0$, then we see that $a_1 a_2 \dots a_n$ is a $\tau$-sequence of
length $n$ if, and only if, $a_1 a_2 \dots a_{n - 1}$ is a $\tau$-sequence of
length
$n - 1$. Hence $\tau(n, 0, \alpha) = \tau(n - 1, 0, \alpha - 1)$ for all
$\alpha > 1$. We may repeatedly apply this observation to show that
$\tau(n, 0, \alpha) = \tau(n - \alpha, 0, 0)$.
Hence we have
$\tau(n, 0, n - i) = \tau(n - (n - i), 0, 0) = \tau(i, 0, 0) = i - 1$.
\end{proof}

\begin{proposition}
For $1 \leq i \leq n$, $\tau(n, n - i) = \frac{1}{2}(i^2 - i + 2)$.
\end{proposition}

\begin{proof}
We proceed by induction on $i$. We start with the case $i = 1$.
To calculate $\tau(n, n - 1)$, let $a_1 a_2 \dots a_n$ be a $\tau$-sequence with
$a_1 = n - 1$. By \ref{tau-rotate} we equivalently have $a_2 a_3 \dots a_n a_1$
is a $\tau$-sequence. Now we may repeatedly apply
\ref{tau-at-least-one} and \ref{tau-descending} to show $a_i = i - 1$, and
that there is a unique possible $\tau$-sequence. Hence $\tau(n, n - 1) = 1$.

Now take $i = k$, with the hypothesis given for $i = k - 1$.
Let $a_1 a_2 \dots a_n$ be a $\tau$-sequence with $a_1 = (n - k)$.
By \ref{tau-descending} we have $a_2 = 0$ or $a_2 = (n - (k - 1))$.
In the first case, we may rotate to get a $\tau$ sequence beginning with 0 and
ending with $n - k$. In the second case, we may rotate to get a
$\tau$-sequence beginning with $n - (k - 1)$. This gives us
\begin{align*}
	\tau(n, n - k) &= \tau(n, 0, n - k) + \tau(n, n - (k - 1)) \\
		&= k - 1 + ((k - 1)^2 - (k - 1) + 2)/2 \\
		&= (k^2 - k + 2)/2. \qedhere
\end{align*}
\end{proof}

We now define a \textit{$\sigma$-sequence} as a sequence of $n = 2k + 1$
integers $a_1 a_2 \dots a_n$ such that
\begin{enumerate}[label=\textnormal{(\roman*)}]
	\item \label{sigma-at-least-zero}
		$a_i \geq 0$,
	\item \label{sigma-one-location}
		if $a_i = 0$ then $a_{i + (k + 1)} = 1$,
	\item \label{sigma-one-location-two}
		if $a_i = 1$ then either $a_{i - 1} = 0$ or
		$a_{i + k} = 0$,
	\item \label{sigma-descending}
		if $a_i > 1$ then $a_{i - 1} = a_i - 1$,
	\item \label{sigma-is-three-seqs}
		there are at most three $i$ such that $a_i = 0$,
	\item \label{sigma-rotate}
		if $a_1 a_2 \dots a_n$ is a $\sigma$-sequence then
		$a_n a_1 a_2 a_3 \dots a_{n - 1}$ must also be a
		$\sigma$-sequence.
\end{enumerate}
As this definition is not as readily visualisable as that of $\tau$-sequences,
we give as examples each sequence for $n \in \{ 9, 11 \}$ without rotations.
{ 
\newcommand{\f}[1]{\textcolor{red}{#1}}
\newcommand{\g}[1]{\textcolor{blue}{#1}}
\newcommand{\h}[1]{\textcolor{green}{#1}}
\begin{table}[H]
\centering
\begin{tabular}{l|l}
$n = 9$   & $n = 11$    \\ \hline
\f{01}234\f{1}234                 & \f{01}2345\f{1}2345 \\
\f{01}\g{01}2\f{1}2\g{1}2         & \f{01}\g{01}23\f{1}2\g{1}23 \\
\f{01}2\g{01}\f{1}23\g{1}         & \f{01}2\g{01}2\f{1}23\g{1}2 \\
\f{001}23\f{11}23                 & \f{01}23\g{01}\f{1}234\g{1} \\
\f{001}\g{01}\f{11}2\g{1}         & \f{001}234\f{11}234 \\
\f{001}\g{1}2\f{11}\g{01}         & \f{001}\g{01}2\f{11}2\g{1}2 \\
\f{01}\h{1}\g{01}\f{1}\h{01}\g{1} & \f{001}2\g{01}\f{11}23\g{1} \\
\f{0001}2\f{111}2                 & \f{001}\g{1}23\f{11}\g{01}2 \\
                                  & \f{001}2\g{1}2\f{11}2\g{01} \\
                                  & \f{01}\g{01}\h{01}\f{1}2\g{1}2\h{1} \\
                                  & \f{01}\h{1}\g{01}2\f{1}\h{01}\g{1}2 \\
                                  & \f{0001}23\f{111}23 \\
\end{tabular}
\end{table}
} 
In our table we have highlighted a pattern made by the 0s and 1s in these
sequences which we aim to formalise and prove. The patterns of 0s and 1s are of
the following forms
\[
	01\underbrace{\dots}_{k - 1}1\underbrace{\dots}_{k - 1},
	\quad
	001\underbrace{\dots}_{k - 2}11\underbrace{\dots}_{k - 2}
	\quad \text{or} \quad
	0001\underbrace{\dots}_{k - 3}111\underbrace{\dots}_{k - 3}.
\]
We shall call these patterns \textit{01-groups}. We aim to show that
each 0 or 1 in a $\sigma$-sequence occurs in a unique 01-group.
In the following, suppose that $a_1 a_2 \dots a_n$ is a $\sigma$-sequence with
$a_1 = 0$ and $a_n \neq 0$.

\begin{lemma}
\label{01_group_from_zero}
There is some $1 \leq \alpha \leq 3$ such that for $1 \leq i \leq \alpha$
we have $a_i = 0, a_{i + (k + 1)} = 1$ and $a_{\alpha + 1} = 1$.
\end{lemma}

\begin{proof}
We let $\alpha$ be the largest number such that $a_i = 0$ for
$1 \leq i \leq \alpha$. From \ref{sigma-is-three-seqs} we see that
$\alpha \leq 3$.
Hence, combining \ref{sigma-at-least-zero},
\ref{sigma-descending} and our definition of $\alpha$, we see that
$0 \leq a_{\alpha + 1} < 2$ and $a_{\alpha + 1} \neq 0$, hence we must have
$a_{\alpha + 1} = 1$. Finally, we may apply \ref{sigma-one-location} and the
fact $a_i = 0$ to show $a_{i + (k + 1)} = 1$.
\end{proof}

\begin{corollary}
Every 0 in a $\sigma$-sequence is in a unique 01-group.
\end{corollary}

\begin{proof}
Consider a $\sigma$-sequence with $a_i = 0$ for some $i$. Using
\ref{sigma-rotate} we may consider a rotation of this sequence which moves
this 0 from $i$ to 1, and then possibly further rotate $a_1$ to $a_2$ or $a_3$
until $a_n \neq 0$. Then we may apply the previous lemma.
\end{proof}

\begin{lemma}
Every 1 in a $\sigma$-sequence is in a unique 01-group.
\end{lemma}

\begin{proof}
Applying \ref{sigma-one-location-two}, we have two possibilities if $a_i = 1$.
In the first possibility, $a_{i - 1} = 0$. In this case,
$a_{i - 1 + (k + 1)} = a_{i + k} = 1$ by \ref{sigma-one-location}. Hence, the
second possibility that $a_{i + k} = 0$ is mutually exclusive with the first. In
the first possibility, we may use Lemma~\ref{01_group_from_zero} to find the
01-group from $a_{i - 1}$ which contains $a_i$, and in the second possibility we
may do the same but from $a_{i + k}$ instead of $a_{i - 1}$.
\end{proof}

Let $\sigma(i, n)$ where $n = 2k + 1$ be the number of $\sigma$-sequences
of length $n$ with $a_1 = i$.

\begin{lemma}
$\sigma(k, n) = 2$.
\end{lemma}

\begin{proof}
If $a_1 = k$ in a $\sigma$-sequence, by \ref{sigma-rotate} we may consider a
rotation such that $a_k = k$. Repeatedly applying \ref{sigma-descending} we may
show that for $i \geq 1$ we have $a_i = i$. For $2 \leq i \leq k$, we have
$a_i > 1$, and hence we have a block of $k - 1$ numbers in our sequence not in a
01-group. Hence, the only possible 01-group in the sequence is
$01\underbrace{\dots}_{k - 1}1\underbrace{\dots}_{k - 1}$. As $a_1 = 1$, and
each occurence of 1 is in a 01-group, we must have this 01-group in our
sequence and no other 01-group may be in this sequence.
We may now consider rotating our $\sigma$-sequence again so that $a_1 = 0$. Now
we may apply \ref{sigma-descending} to show $a_i = i$ and $a_{i + (k + 1)} = i$
for all $2 \leq i \leq k$, and this is the only $\sigma$-sequence containing $k$
up to rotation. Finally, we may rotate this sequence in two ways to make
$a_1 = k$, hence $\sigma(k, n) = 2$.
\end{proof}

\begin{lemma}
$\sigma(0, n) \geq 3$.
\end{lemma}

\begin{proof}
For $k \geq 3$ we consider the $\sigma$-sequence with $a_1 = a_2 = a_3 = 0$,
$a_4 = a_{k + 2} = a_{k + 3} = a_{k + 4} = 1$ and all other $a_i$ filled in
using \ref{sigma-descending}. This sequence may be rotated to give $a_1 = 0$ in
three different ways, hence in this case $\sigma(0, n) \geq 3$.

For $k = 2$, we consider the $\sigma$-sequences 00111, 01110 and 01212 to see
$\sigma(0, n) \geq 3$.
\end{proof}

\begin{lemma}
$\sigma(i, n) < \sigma(i - 1, n)$ for $1 < i \leq k$.
\end{lemma}

\begin{proof}
Consider the map $\phi$ which takes a $\sigma$-sequence $a_1 a_2 \dots a_n$ to
$a_n a_1 a_2 \dots a_{n - 1}$. If $a_1 = i$, then $a_n = i - 1$ by
\ref{sigma-descending}, hence $\phi$ is an injective map from $\sigma$-sequences
starting with $i$ to those starting with $i - 1$. Hence, to show
$\sigma(i, n) < \sigma(i - 1, n)$ we need only find a $\sigma$-sequence with
$a_1 = i - 1$ and $a_2 \neq i$.

For $i \leq k - 1$, the sequence with
$a_1 = 0, a_2 = a_{k + 2} = 1, a_{i + 1} = 0, a_{i + 2} = a_{i + k + 2} = 1$ and
all other $a_j$ satisfying $a_j = a_{j - 1} + 1$ is a $\sigma$-sequence with
$a_i = i - 1$ and $a_{i + 1} = 0 \neq i$. Hence, we can take a rotation of this
by \ref{sigma-rotate} to find a $\sigma$-sequence with $a_1 = i - 1$ and
$a_2 \neq i$.

Now, for $i = k$, we consider the sequence
$a_1 = 0, a_2 = a_{k + 2} = 1, a_{k - 1} = 0, a_k = 1$ and $a_n = 1$, and all
other $a_j$ satisfying $a_j = a_{j - 1} + 1$. We see this is a $\sigma$-sequence
with $a_{n - 1} = i - 1$ and $a_n = 1 \neq i$. Hence, again we can take a
rotation of this by \ref{sigma-rotate} to find a $\sigma$-sequence with
$a_1 = i - 1$ and $a_2 \neq i$.
\end{proof}

\section{Word Graphs}

To facilitate our discussion of the G\'{o}mez graphs we first introduce the
notion of a \textit{word graph} and word graph \textit{families},
which form a natural generalisation of the
construction of the G\'{o}mez and Faber-Moore-Chen graphs
(found in \cite{gomez} and \cite{faber-moore-chen} respectively).

To define a word graph, fix some number $n$, the \textit{word length},
some set $\Pi_n \subseteq S_n$, the \textit{rules}, and
some $m > n$, the \textit{alphabet size}. We define the word graph
$G_m = \langle V, E \rangle$ as follows. Fix some arbitrary set $B$
such that $|B| = m$, let
$V = \{ x_1 x_2 \dots x_n | x_i \in B, x_i = x_j \Leftrightarrow i = j \}$, that
is the vertices of $G_m$ are the words of length $n$ on $B$ all of whose letters
are distinct, and we form the directed adjacencies of $G_m$ by the following
rules
\[
	x_1 x_2 \dots x_n \rightarrow \begin{cases}
		x_2 x_3 \dots x_n y,
			& y \in B \setminus \{ x_1, x_2, \dots, x_n \}, \\
		x_{\pi(1)} x_{\pi(2)} \dots x_{\pi(n)}, & \pi \in \Pi_n.
	\end{cases}
\]
We define the word graph family of $\Pi_n$ to be
$\{ G_{n + 1}, G_{n + 2}, \dots \}$
and denote it by $\mywgp{\Pi_n}$. For the following, let $\Pi_n$ be an arbitrary
rule set and $G_m \in \mywgp{\Pi_n}$.

We will refer to the rules of the form
$x_1 x_2 \dots x_n \rightarrow x_2 x_3 \dots x_n y$ for $y \not\in \{ x_i \}$ as
\textit{alphabet changing} and rules of the form
$x_1 x_2 \dots x_n \rightarrow x_{\pi(1)} x_{\pi(2)} \dots x_{\pi(n)}$ as
\textit{alphabet fixing}. For a vertex $v = x_1 x_2 \dots x_n$, we shall define
$\alpha$ by $\alpha(v) = \{ x_1, x_2, \dots, x_n \}$ and refer to $\alpha(v)$ as
the \textit{alphabet} of $v$.

\begin{lemma}
\label{all_at_least_n}
For all $m \geq 2n$ we have $\mydiam{G_m} \geq n$.
\end{lemma}

\begin{proof}
Letting $B = \{ x_1, x_2, \dots, x_n, y_1, y_2, \dots, y_n \}$ we consider any
path from $u = x_1 x_2 \dots x_n$ to $v = y_1 y_2 \dots y_n$. As each rule of
$G_m$ introduces at most one new letter which is not in $\alpha(u)$, and each
letter of $\alpha(v)$ is not in $\alpha(u)$, we must have at least
$|\alpha(v)| = n$ rules in a path from $u$ to $v$.
\end{proof}

\begin{lemma}
\label{path_of_double_length}
For all $m \geq 3n$ we have $\mydiam{G_m} \leq 2n$.
\end{lemma}

\begin{proof}
Consider $u = x_1 x_2 \dots x_n, v = y_1 y_2 \dots y_n \in \myvert{G_m}$, and
let
$\{ z_1, z_2, \dots, z_n \} \subseteq B \setminus (\alpha(u) \cup \alpha(v))$.
Letting $w = z_1 z_2 \dots z_n$, we can create a path of length $n$ from $u$ to
$w$ by using the alphabet changing rule to append $z_i$ at the
$i$\textsuperscript{th} step in the path. This is always possible as
$\alpha(u) \cap \alpha(w) = \emptyset$. We then may form another such path of
length $n$ from $w$ to $v$ by the same logic. Concatenating these two paths
gives us a path of length $2n$ from $u$ to $v$.
\end{proof}

\begin{lemma}
\label{eventual_diameter}
For all $m \geq 4n$ we have $\mydiam{G_m} = \mydiam{G_{4n}}$.
\end{lemma}

\begin{proof}
Take arbitrary $u, v \in G_m$. Lemma~\ref{path_of_double_length} tells us that
$\mydist{u}{v} \leq 2n$, hence if we consider a shortest path connecting $u$ and
$v$ we know it is of length at most $2n$. On such a path, whenever we encounter
alphabet changing rules, denote by $z_i$ the element introduced by the
$i$\textsuperscript{th} alphabet changing rule. Now let
$B' = \alpha(u) \cup \alpha(v) \cup \{ z_i \}$, note that we have
$|B'| \leq 4n$. Hence, we see that the shortest path connecting $u$ and $v$ is
in the subgraph $H$ induced by the vertices
$\{ x_1 x_2 \dots x_n | x_i \in B' \} \subseteq \myvert{G_m}$, which is
trivially isomorphic to a subgraph of $G_{4n}$. As $G_{4n}$ is also a subgraph
of $G_m$, the result immediately follows.
\end{proof}

From Lemma~\ref{eventual_diameter}, for a given $\Pi_n$ we call the
$\mydiam{G_{4n}}$ the \textit{eventual diameter} of $\mywgp{\Pi_n}$, and note
from Lemma~\ref{all_at_least_n} that the eventual diameter is at least $n$.

\begin{proposition}
A family of word graphs, $\mywgp{\Pi_n}$, is asymptotically close to the Moore
bound if, and only if, its eventual diameter is $n$.
\end{proposition}

\begin{proof}
Let $G_m \in \mywgp{\Pi_n}$ for $m \geq 4n$. Suppose the eventual diameter of
$\mywgp{\Pi_n}$ is $n + \varepsilon$. The degree of $G_m$ is given by
$|\Pi_n| + (m - n)$, hence letting $\alpha = |\Pi_n| - n$ we have the
$\mydeg{G_m} = m + \alpha$. Finally, we may count the size of $G_m$ as follows
\[
	|\myvert{G_m}| = n! \binom{m}{n}
		= \frac{m!}{(m - n)!} = m^n + \mybigo{m^{n - 1}}.
\]
Now we recall the Moore bound for a directed graph of degree $d$ and diameter
$k$ is given by
$\mymoore{d}{k} = d^k + d^{k - 1} + \dots + 1 = d^k + \mybigo{d^{k - 1}}$.
Hence, letting $d_m = \mydeg{G_m}$ and $k_m = \mydiam{G_m}$, we have
\begin{align*}
	\lim_{m \to \infty} \left\{
			\frac{|\myvert{G_m}|}{\mymoore{d_m}{k_m}} \right\}
		&= \lim_{m \to \infty} \left\{
			\frac{m^n + \mybigo{m^{n - 1}}}
			{\mymoore{m + \alpha}{n + \varepsilon}} \right\} \\
		&= \lim_{m \to \infty} \left\{
			\frac{m^n + \mybigo{m^{n - 1}}}
			{ (m + \alpha)^{n + \varepsilon}
			+ \mybigo{m^{n + \varepsilon - 1}} } \right\} \\
		&= \lim_{m \to \infty} \left\{
			\frac{m^n + \mybigo{m^{n - 1}}}
			{ m^{n + \varepsilon}
			+ \mybigo{m^{n + \varepsilon - 1}} } \right\} \\
		&= \begin{cases}
			1, & \text{if $\varepsilon = 0$}, \\
			0, & \text{otherwise.}
		\end{cases} \qedhere
\end{align*}
\end{proof}

Hence, we now introduce the restriction that a rule set $\Pi_n$ is
\textit{admissible} if the eventual diameter of $\mywgp{\Pi_n}$ is $n$. For the
rest of this section, we will only consider admissible sets $\Pi_n$.

We also now introduce the further restriction that, for all $\pi \in \Pi_n$,
$\pi(i) \leq i + 1$. Informally this means that the alphabet fixing rules cannot
``shift'' any letter to the left more than one space at a time.
We call this \textit{shift restriction} and note that the
G\'{o}mez and Faber-Moore-Chen graphs are shift restricted word graphs.
For the remainder of this section, we will only consider shift restricted word
graphs. We now show that the G\'{o}mez graphs are largest possible shift
restricted word graphs for given degree and diameter.

Let $\Pi_n$ be admissible and shift restricted, and let $G_m \in \mywgp{\Pi_n}$,
where $m > n$.
For each vertex $v \in \myvert{G_m}$ and letter $x \in B$ we introduce the function
$p_x(v)$ which is the position of the letter $x$ in $v$. The function is defined
by $p_{x_i}(x_1 x_2 \dots x_n) = i$ and $p_y(x_1 x_2 \dots x_n) = 0$ where
$y \not\in \{ x_1, x_2, \dots, x_n \}$.

\begin{lemma}
\label{shift_one}
For $u, v \in V$ with $u \rightarrow v$ and $y \in B$, we have
$p_y(v) \geq p_y(u) - 1$.
\end{lemma}

\begin{proof}
If $p_y(u) = 0$ then the result is immediate as $p_y(v') \geq 0$ for all
$v' \in V$. Hence, suppose $u = x_1 x_2 \dots x_n$ and $y = x_i$. If
$v = x_2 x_3 \dots x_n y$, then $p_{x_i}(v) = p_{x_i}(u) - 1$. If
$v = x_{\pi(1)} x_{\pi(2)} \dots x_{\pi(n)}$ and $\pi(j) = i$ then
$p_{x_i}(v) = p_{x_j}(u) = j \geq \pi(j) - 1 = p_{x_i}(u) - 1$.
\end{proof}

\begin{corollary}
\label{path_len_pos}
For any $u, v \in V$ and $y \in B$, all paths connecting $u$ to $v$ have length
at least $p_y(u) - p_y(v)$ (if $p_y(u) = 0$ and $p_y(v) > 0$ this becomes
$n + 1 - p_y(u)$).
\end{corollary}

\begin{lemma}
\label{at_least_diam_n}
$\mydiam{G_m} \geq n$.
\end{lemma}

\begin{proof}
Let $u = x_1 x_2 \dots x_n$ and $v = y x_1 x_2 \dots x_{n - 1}$. We have
$p_{x_n}(u) = n$ and $p_{x_n}(v) = 0$, hence any path connecting $u$ and $v$ is
at least length $n$.
\end{proof}

\begin{proposition}
$\mydiam{G_m} = n$.
\end{proposition}

\begin{proof}
By our assumption of the eventual diameter being $n$, we need only show this for
$m < 4n$. Hence, consider $u, v \in \myvert{G_m}$ where $m < 4n$.
Let $\phi : G_m \hookrightarrow G_{4n}$ be the inclusion from $G_m$ to
$G_{4n}$, and consider a path from $u' = \phi(u)$ to $v' = \phi(v)$ in $G_{4n}$.
Letting $B' = \alpha(u') \cup \alpha(v')$, we can see that any vertex
$w \in G_{4n}$ satisfying $\alpha(w) \subseteq B'$ is invertible by $\phi$.
Suppose that on a path from $u'$ to $v'$ we introduce a letter $y \not\in B'$
via an alphabet changing rule, call the vertex after this rule $w$.
We have $p_y(w) = n$ and $p_y(v') = 0$, hence the remainder of the path is at
least length $n$, but this contradicts that the path is of length at most $n$.
Therefore all vertices $w$ between $u'$ and $v'$ satisfy
$\alpha(w) \subseteq B'$. Hence, we may use $\phi^{-1}$ to find a path
connecting $u$ and $v$ of length at most $n$ in $G_m$. This shows
$\mydiam{G_m} \leq n$, applying Lemma~\ref{at_least_diam_n} gives us the result.
\end{proof}

\begin{corollary}
If $\Pi_n$ and $\Tau_n$ are admissible and shift restricted, and
$|\Pi_n| < |\Tau_n|$, then each $G_m \in \mywgp{\Pi_n}$ and
$H_m \in \mywgp{\Tau_n}$ have the same diameter and are the same size for all
$m$, but $\mydeg{G_m} < \mydeg{H_m}$.
\end{corollary}

Hence we now may make the definition that an admissible shift restricted
rule set $\Pi_n$ is \textit{optimal}
if there exists no rule set $\Tau_n$ with $|\Tau_n| < |\Pi_n|$ which is also
both admissible and shift restricted.

\begin{lemma}
\label{contains_cycle}
For each $1 \leq i \leq n$ there exists some $\pi \in \Pi_n$ which contains an
$i$-cycle.
\end{lemma}

\begin{remark}
The proof of this lemma obscures that this is a fairly natural
observation. We illustrate below an example path used in the lemma,
noting that the key idea is that the letter $x_n$ has to ``jump''
the block of $y_i$s by exactly $k + 1$ spaces. At the point of the jump, we must
have a permutation $\pi \in \Pi_n$, and the jump is a $(k + 1)$-cycle in that
permutation.
\[
	\newcommand{\f}[1]{\textcolor{red}{#1}}
	\newcommand{\g}[1]{\textcolor{blue}{#1}}
	\begin{matrix}
	x_1 & x_2 & x_3 & x_4 & x_5 & x_6 & x_7 & x_8 & \f{x_9} \\
	x_2 & x_3 & x_4 & x_5 & x_6 & x_7 & x_8 & \f{x_9} & \g{y_1} \\
	x_3 & x_4 & x_5 & x_6 & x_7 & x_8 & \f{x_9} & \g{y_1} & \g{y_2} \\
	x_4 & x_5 & x_6 & x_7 & x_8 & \f{x_9} & \g{y_1} & \g{y_2} & \g{y_3} \\
	x_5 & x_6 & x_7 & x_8 & \f{x_9} & \g{y_1} & \g{y_2} & \g{y_3} & x_1 \\
	x_6 & x_7 & x_8 & \f{x_9} & \g{y_1} & \g{y_2} & \g{y_3} & x_1 & x_2 \\
	x_7 & x_8 & \f{x_9} & \g{y_1} & \g{y_2} & \g{y_3} & x_1 & x_2 & x_3 \\
	x_7 & x_8 & \g{y_1} & \g{y_2} & \g{y_3} & \f{x_9} & x_1 & x_2 & x_3 \\
	x_8 & \g{y_1} & \g{y_2} & \g{y_3} & \f{x_9} & x_1 & x_2 & x_3 & x_4 \\
	\g{y_1} & \g{y_2} & \g{y_3} & \f{x_9} & x_1 & x_2 & x_3 & x_4 & x_5 \\
	\end{matrix}
\]
\end{remark}

\begin{proof}
For $k \geq 1$ we show the existence of a permutation containing a
$(k + 1)$-cycle, noting that when $k + 1 = n - 1$ the remaining fixed element of
the permutation provides the missing ``1-cycle''.
Let $u = x_1 x_2 \dots x_n$ and
$v = y_1 y_2 \dots y_k x_n x_1 x_2 \dots x_{n - (k + 1)}$.
As $p_{y_1}(u) = 0$ and $p_{y_1}(v) = 1$, the shortest path that connects $u$
and $v$ is of length $n$. Hence, let
$u = u_0 \rightarrow u_1 \rightarrow \dots \rightarrow u_n = v$ be such a path.
A trivial induction and Lemma~\ref{shift_one} shows for $1 \leq j \leq k$ we
have $p_{y_j}(u_{i + j}) = n - i$. A similar induction starting from $u_0$
shows that $p_{x_n}(u_i) \geq n - i$ and an induction working backwards from
$u_n$ shows that $p_{x_n}(u_i) \leq n - i + k + 1$. As eventually we have
$p_{x_n}(u_n) = k + 1 > 0$ we know there exists $c$ such that
$p_{x_n}(u_{c - 1}) = n - c + 1$ and $p_{x_n}(u_c) > n - c$.
Now consider the first such $c$. For each $1 \leq j \leq k$ we cannot have
$p_{x_n}(u_c) = n - c + j$, as we have that $p_{y_j}(u_c) = n - c + j$, and we
cannot have some $j > k + 1$ such that $p_{x_n}(u_c) = n - c + j$, as we have
$p_{x_n}(u_c) \leq n - c + k + 1$. The only possibility this leaves is that
$p_{x_n}(u_c) = n - c + k + 1$. This can only happen if the edge connecting
$u_{c - 1}$ and $u_c$ is alphabet fixing, corresponding to some rule $\pi$.

We see that in $\pi$, each $y_j$ went from $n - c + i + j + 1$ to
$n - c + i + j$, and $x_n$ went from $n - c + 1$ to $n - c + k + 1$, hence,
letting $\alpha = n - c + 1$, we have the $k + 1$-cycle
$((\alpha + k) \  (\alpha + k - 1) \  \dots \  \alpha)$.
\end{proof}

\begin{corollary}
The G\'{o}mez graphs are optimal.
\end{corollary}

\noindent
\textit{(The reader unfamiliar with the definition of the G\'omez graphs will
find the definition at the beginning of Section~\ref{intro_to_gomez}).}

\begin{proof}
For $n = 2k + 1$, the set $\Pi_n$ which defines the G\'{o}mez graphs contains
exactly one cycle of each length $1 \leq i \leq n$.

For $n = 2k$, the set $\Pi_n$ which defines the G\'{o}mez graphs contains
exactly one cycle of each length $1 \leq i \leq n$, $i \neq k$, and two cycles
of length $k$. As each permutation is a permutation on $n = 2k$ elements, it is
not possible to remove a permutation from $\Pi_n$ by eliminating only one
$k$-cycle.
\end{proof}

Altogether, this shows that if we want to try to create new word graphs which
are larger than G\'{o}mez graphs for a given degree and diameter, then we will
either have to consider non-admissible $\Pi_n$ to find small examples, which
would be limited to $m < 2n$, or consider word graphs which are
not shift restricted.

We will now proceed in this section by establishing other properties shared by
word graphs. In particular, we are interested in when they are Cayley, and shall
provide a table of which values of $n$ and $m$ correspond to Cayley graphs. We
note that it is possible that other values of $n$ and $m$ can correspond to
Cayley graphs also, but this can only happen if the graph
$G_m \in \mywgp{\Pi_n}$ contains an automorphism outside of $S_m$.

\begin{lemma}
\label{symmetric_group_lemma}
There is some group $H \leq \myaut{G_m}$ with $H \cong S_m$.
\end{lemma}

\begin{proof}
We construct $H$ by taking all $\phi \in S_m$ acting naturally on $B'$ and
defining $\phi' : \myvert{G_m} \to \myvert{G_m}$ by
$\phi'(x_1 x_2 \dots x_n) = \phi(x_1) \phi(x_2) \dots \phi(x_n)$.
\end{proof}

In light of this lemma, we quote a result available in
\cite{regular-action, regular-action-2, regular-action-3}
and used in \cite{when-cayley} and use it to classify when
word graphs are Cayley. The following is an exhaustive table of
values of $n$ and $m$ such that there is a subgroup of $S_m$ acting regularly on
the tuples of length $n$, and the subgroups which have this action.

\begin{table}[H]
\centering
\begin{tabular}{|l|l|l|}
\hline
$n$ & $m$     & Group                             \\ \hline
$k$ & $k$     & $S_k$                             \\
$k$ & $k + 1$ & $S_{k + 1}$                       \\
$k$ & $k + 2$ & $A_{k + 2}$                       \\
2   & $q$     & Finite near-field                 \\
3   & $q + 1$ & $\mypsl{2}{q}, \mypdeltal{2}{q} $ \\
4   & 11      & $M_{11}$                          \\
5   & 12      & $M_{12}$                          \\ \hline
\end{tabular}
\end{table}

\begin{corollary}
For $n, m$ in this table, the graph $G_m \in \mywgp{\Pi_n}$ is Cayley.
\end{corollary}

\begin{corollary}
If $\myaut{G_m} \cong S_m$, then the graph $G_m \in \mywgp{\Pi_n}$ is Cayley if,
and only if, $n$ and $m$ are in the given table.
\end{corollary}

Hence, we now conclude this section by establishing a test to determine whether
a given family of word graphs satisfies $\myaut{G_m} \cong S_m$. First we shall
need some definitions. Let $\Gamma_n$ be the
Cayley graph $\mycayley{\Pi_n}{S_n}$.

\begin{lemma}
Letting $H$ be a subgraph of $G_m$ induced by vertices
$\{ x_1, x_2, \dots, x_n \} \subset B$, we have $H \cong \Gamma_n$.
\end{lemma}

\begin{proof}
This is simply a relabelling.
\end{proof}

We will now refer to the graph $\Gamma_n$ as the
\textit{alphabet fixing subgraph} of $G$, noting that it is unique to
isomorphism regardless of the choice of $\{ x_i \}$.
Now we make two further definitions. We shall call a word graph
$G$ \textit{alphabet stable} if there exists no automorphism $\phi \in \myaut{G}$
such that there exist some $u, v \in \myvert{G}$ with $\alpha(u) = \alpha(v)$
but $\alpha(\phi(u)) \neq \alpha(\phi(v))$. In other words, a word graph is
alphabet stable if, and only if, it preserves whether arcs are alphabet changing
or alphabet fixing. Second, we shall call a family of word graphs
\textit{subregular} if the alphabet fixing subgraph $\Gamma_n$ of $G_m$ is
regular, i.e. $\myaut{\Gamma_n} \cong S_n$. In the following let $G_m$ be a
word graph which is alphabet stable and subregular. We now aim to show that
$\myaut{G_m} \cong S_m$.

\begin{lemma}
\label{fixes_alpha_subs}
If $\phi \in \myaut{G_m}$ fixes a vertex $u$, then $\phi$ fixes all $v$ such
that $\alpha(u) = \alpha(v)$.
\end{lemma}

\begin{proof}
Let $V = \{ v \in \myvert{G_m} | \alpha(v) = \alpha(u) \}$. Consider
$\psi = \phi|_V$, the restriction of $\phi$ to the vertices of $V$. For any
$v \in V$ we have $\alpha(\psi(v)) = \alpha(\phi(v)) = \alpha(u)$, hence we have
$\psi(v) \in V$. As $\phi$ is an automorphism, $\psi$ is injective, and
therefore bijective as its image is its domain. Hence $\psi$ is a bijection on
the subgraph induced by vertices of $V$, which is the alphabet fixing graph
$\Gamma_n$. As $G_m$ is subregular, any automorphism of $\Gamma_n$ which fixes a
vertex must fix all of $\Gamma_n$. Therefore, as $\psi(u) = u$ we must have that
$\psi$ is the identity on $V$.
\end{proof}

\begin{lemma}
\label{inductive_step}
If $\phi \in \myaut{G_m}$ and $X, Y, Z \subset B$ with the following properties
\begin{itemize}
	\item $X = \{ x_1, x_2, z_1, z_2, \dots, z_{n - 2} \}$,
	\item $Y = \{ y_1, y_2, z_1, z_2, \dots, z_{n - 2} \}$,
	\item $Z = \{ x_2, y_2, z_1, z_2, \dots, z_{n - 2} \}$,
	\item $\phi$ fixes all $v \in \myvert{G_m}$ with $\alpha(v) = X$ or
		$\alpha(v) = Y$,
\end{itemize}
then $\phi$ fixes all $v \in \myvert{G_m}$ with $\alpha(v) = Z$.
\end{lemma}

\begin{proof}
Let $u = x_1 z_1 z_2 \dots z_{n - 2} x_2, v = y_1 z_1 z_2 \dots z_{n - 2} y_2$
and suppose we have $w, w' \in \myvert{G_m}$ such that
$u \rightarrow w, v \rightarrow w'$ and $\alpha(w) = \alpha(w')$.
As $|X \cap Y| = 2$, we must have that both $u \rightarrow w$ and
$v \rightarrow w'$ are alphabet changing rules. Therefore,
$x_1 \not\in \alpha(w), x_2 \in \alpha(w), y_1 \not\in \alpha(w')$ and
$y_2 \in \alpha(w')$. Hence we have
$\alpha(w) \supseteq (X \cap Y) \cup \{ x_2, y_2 \}$, but
$|\alpha(w)| = |(X \cap Y) \cup \{ x_2, y_2 \}|$, and hence we have equality.
We now see the rule $u \rightarrow w$ must introduce the letter $y_2$, and
$w = z_1 z_2 \dots z_{n - 2} x_2 y_2$. Similarly,
$w' = z_1 z_2 \dots z_{n - 2} y_2 x_2$. By our assumptions on $\phi$ we have
$\phi(u) = u$, $\phi(v) = v$, and $\alpha(\phi(w)) = \alpha(\phi(w'))$ as
$\alpha(w) = \alpha(w')$ and $G_m$ is alphabet stable. Hence we have
$u \rightarrow \phi(w)$ and $v \rightarrow \phi(w')$ with
$\alpha(\phi(w)) = \alpha(\phi(w'))$, so $\phi(w) = w$ and $\phi(w') = w'$.
Now applying Lemma~\ref{fixes_alpha_subs} we get the desired result.
\end{proof}

\begin{lemma}
\label{fixer_is_identity}
The only $\phi \in \myaut{G_m}$ which fixes a vertex $u \in \myvert{G_m}$ and
all $v \in \myvert{G_m}$ such that $u \rightarrow v$ is the identity.
\end{lemma}

\begin{proof}
We may label $u$ as $x_1 x_2 \dots x_n$ taking $B$ to be
$\{ x_1, x_2, \dots, x_n, y_1, y_2, \dots, y_{m - n} \}$. For a vertex $v$,
define $f(v) = |\{ x_1, x_2, \dots, x_n \} \cap \alpha(v)|$. We show by
induction for $n \geq k \geq 0$ that $\phi$ fixes all $v$ such that $f(v) = k$.

For $k = n$, take any $v$ with $f(v) = n$. We have $\alpha(v) = \alpha(u)$, and
$u$ is fixed by $\phi$. Hence by Lemma~\ref{fixes_alpha_subs} $\phi$ fixes $v$
also.

For $k = n - 1$, we have $f(v) = n - 1$. First we consider
$\alpha(v) = \{ x_2, x_3, \dots, x_n, y \}$. In this case the vertex
$v' = x_2 x_3 \dots x_n y$ satisfies $u \rightarrow v'$ and so $\phi(v') = v'$.
Hence, as $\alpha(v) = \alpha(v')$ we use Lemma~\ref{fixes_alpha_subs} again to
show that $\phi(v) = v$. For other $v$ with $f(v) = n - 1$, without loss of
generality we may assume $\alpha(v) = \{ x_1, x_2, \dots, x_{n - 1}, y \}$.
Applying Lemma~\ref{inductive_step} to the sets $\{ x_1, x_2, \dots, x_n \}$ and
$\{ x_2, x_3, \dots, x_n, y \}$ we see that $v$ is fixed.

For $k = c$ given the inductive hypothesis for $k = c + 1$, let
$v \in \myvert{G_m}$ such that
\[ \alpha(v) = \{ x_1, \dots, x_c, y_1, \dots, y_{n - c} \}. \]
By applying the
inductive hypothesis and Lemma~\ref{fixes_alpha_subs} to the sets
\[
	\{ x_1, \dots, x_{c + 1}, y_1, \dots, y_{n - c - 1} \}
	\quad \text{and} \quad
	\{ x_1, \dots, x_{c + 1}, y_2, \dots, y_{n - c} \}
\]
we get the desired result.
\end{proof}

\begin{proposition}
$\myaut{G_m} \cong S_m$.
\end{proposition}

\begin{proof}
Let $H \leq \myaut{G_m}$ as defined in Lemma~\ref{symmetric_group_lemma}.
Suppose $\phi \in \myaut{G_m}$. Consider some $u \in \myvert{G_m}$, and define
$\psi \in H$ such that $\psi(\phi(u)) = u$ and for all $v \in \myvert{G_m}$ with
$u\rightarrow v$ via an alphabet changing rule we have $\psi(\phi(v)) = v$. Note
that $\psi$ is guaranteed to exist as alphabet stability guarantees this process
of defining $\psi$ corresponds to defining a unique permutation. We now consider
the automorphism $\psi \circ \phi$, which by Lemma~\ref{fixer_is_identity} must
be the identity. Hence $\psi = \phi^{-1}$ and $\phi \in H$.
\end{proof}

Now we see that alphabet stability and subregularity are sufficient conditions
to guarantee $\myaut{G_m} \cong S_m$, we devote the remainder of this section to
creating tests to determine when a family of word graphs is alphabet stable and
subregular. Our tests will only concern the counting of certain paths in the
alphabet fixing subgraph. In the following we consider the word graph
$G_m \in \mywgp{\Pi_n}$ with alphabet fixing subgraph $\Gamma_n$.

\begin{lemma}
\label{unique_n_path}
If $u, v \in \myvert{G_m}$ such that $u \rightarrow v$ and
$\alpha(u) \neq \alpha(v)$, then there is a unique path of length $n$ from $v$
to $u$.
\end{lemma}

\begin{proof}
Without loss of generality we may take $u = x_1 x_2 \dots x_n$ and
$v = x_2 x_3 \dots x_n y$. Considering a path
$u = u_0 \rightarrow u_1 \rightarrow \dots \rightarrow u_k = v$ with $k \leq n$,
we may repeatedly apply Corollary~\ref{path_len_pos} whilst considering
$p_{x_i}(u_{i - 1})$ and $p_{x_i}(u_k)$ to deduce that
$u_{i - 1} \rightarrow u_i$ by the alphabet changing rule which introduces
$x_i$.
\end{proof}

Now suppose for all $u, v \in \myvert{\Gamma_n}$ with $u \rightarrow v$ we have
either more than one path from $v$ to $u$ of length $n$, or we have a path of
length less than $n$ from $v$ to $u$.

\begin{lemma}
\label{no_exchanges}
There is no $\phi \in \myvert{G_m}$ with $\phi(u) = u$ and
$\alpha(\phi(v)) \neq \alpha(u)$.
\end{lemma}

\begin{proof}
If such a $\phi$ exists then $u \rightarrow \phi(v)$ by an alphabet changing
rule. Hence there is a unique shortest path of length $n$ connecting
$\phi(v)$ to $u$.
\end{proof}

\begin{proposition}
$G_m$ is alphabet stable.
\end{proposition}

\begin{proof}
Suppose $G_m$ is not alphabet stable.
Let $\phi \in \myaut{G_m}$ and $u, v \in \myvert{G_m}$ such that
$\alpha(u) = \alpha(v)$ and $\alpha(\phi(u)) \neq \alpha(\phi(v))$. Consider a
path from $u$ to $v$ of length at most $n$, say
$u = u_0 \rightarrow \dots \rightarrow u_k = v$.
Lemma~\ref{unique_n_path} shows that $\alpha(u) = \alpha(u_i)$ for each $i$.
Now consider the path $\phi(u_0) \rightarrow \dots \rightarrow \phi(u_k)$. As
$\alpha(\phi(u_0)) = \alpha(\phi(u)) \neq \alpha(\phi(v)) = \alpha(\phi(u_k))$
there must be some $c$ such that
$\alpha(\phi(u_c)) \neq \alpha(\phi(u_{c + 1}))$. Hence, we have
$u_c \rightarrow u_{c + 1}$ and $\alpha(u_c) = \alpha(u_{c + 1})$, but
$\alpha(\phi(u_c)) \neq \alpha(\phi(u_{c + 1}))$, contradicting
Lemma~\ref{no_exchanges}.
\end{proof}

\begin{lemma}
If $\Gamma_n$ is not regular, there is an automorphism
$\phi \in \myaut{\Gamma_n}$ such that for some
$u, v \in \myvert{\Gamma_n}$ with $u \rightarrow v$
we have $\phi(u) = u$ but $\phi(v) \neq v$.
\end{lemma}

\begin{proof}
Let $H < \myaut{\Gamma_n}$ be regular, and let
$\phi \in \myaut{\Gamma_n} \setminus H$. Consider $u \in \myvert{\Gamma_n}$ and
let $\psi \in H$ be the automorphism such that $\psi(\phi(u)) = u$. Let
$\phi' = \psi \circ \phi$, so $\phi'$ fixes $u$. As $\phi \not\in H$, we must
have $\phi'$ is not the identity. Hence there is some $v \in \myvert{\Gamma_n}$
such that $\phi'(v) \neq v$. Consider a path from $u$ to $v$, we must encounter
a pair of vertices on the path such that $u' \rightarrow v'$, $\phi'(u') = u'$
and $\phi'(v') \neq v'$.
\end{proof}

\begin{corollary}
If for all $u, v, w \in \myvert{\Gamma_n}$ with $u \rightarrow v$ and
$u \rightarrow w$ there exists some $k$ such that the number of paths of length
$k$ from $v$ to $u$ is different to the number of paths of length $k$ from $w$
to $u$, then $\Gamma_n$ is regular.
\end{corollary}

We now combine these results and state our test. Let $G_m \in \mywgp{\Pi_n}$ be
a word graph with alphabet fixing subgraph $\Gamma_n$. Let
$u \in \myvert{\Gamma_n}$ be an arbitrary fixed vertex of $\Gamma_n$ and let
$\{ v_i \}$ be the set of vertices such that $u \rightarrow v_i$.

\begin{proposition}
If the following conditions are satisfied, then $\myaut{G_m} \cong S_m$.
\begin{itemize}
	\item each $v_i, v_j$ has some $k$ such that the number of paths of
		length $k$ from $v_i$ to $u$ is different to the number of paths
		of length $k$ from $v_j$ to $u$.
	\item each $v_i$ has either a path of length less than $n$ to $u$ or has
		more than one path of length $n$ to $u$.
\end{itemize}
\end{proposition}

\section{Introduction to G\'{o}mez Graphs} \label{intro_to_gomez}

In our account of G\'{o}mez graphs, we shall use a modified notation to that of
the original paper more appropriate to our purposes. We note that this paper
will only deal with the G\'{o}mez graphs corresponding to the graphs
$\mydg{k}{k}$ and $\mydg{k}{k + 1}$. The technique used herein does not work for
all $\mydg{k}{k'}$ where $k' \geq k$, which we shall provide explicity examples
to show, and runs into difficulty when pursued for the case $\mydg{k}{k + 2}$.
Hence, we only deal with the cases which provide the extreme examples in
degree-diameter as opposed to dealing with the entire family.

We begin by giving a definition of the alphabet fixing subgraphs $\Gamma_n$
of the G\'{o}mez graphs.
For any $n$, define $k$ so that either $n = 2k + 1$ or $n = 2k$, and let $B$
be any set such that $|B| = n$.
We define the graph $\Gamma_n = \langle V, E \rangle$ as follows. The set $V$ of
vertices is given by
$V = \{ x_1 x_2 \dots x_n | x_i \in B, x_i = x_j \Leftrightarrow i = j \}$,
that is $V$ is the set of all words of length $n$ on the alphabet $B$ with
distinct letters, and the set $E$ is given by the directed adjacencies
\[
	x_1 x_2 \dots x_n \rightarrow \begin{cases}
		x_2 x_3 \dots x_{k - i} x_1
			x_{k - i + 2} x_{k - i + 3} \dots x_n x_{k - i + 1}, &
			\text{for $0 \leq i < k$}, \\
		x_2 x_3 \dots x_n x_1.
	\end{cases}
\]
Informally, each of these rules splits the word into a left and right half and
rotates each half by one. The size of the left half is not allowed to exceed
that of the right, and we also allow an empty left half. Now we give the example
adjacencies for the cases $n = 6$ and $n = 7$ with the left and right halfs
coloured for clarity.
\begin{table}[H]
\centering
\begin{tabular}{c|c}
$n = 6$ & $n = 7$ \\ \hline
$ x_1 x_2 x_3 x_4 x_5 x_6 \rightarrow \begin{cases}
	\textcolor{blue}{x_2 x_3 x_1}\textcolor{red}{x_5 x_6 x_4} \\
	\textcolor{blue}{x_2 x_1}\textcolor{red}{x_4 x_5 x_6 x_3} \\
	\textcolor{blue}{x_1}\textcolor{red}{x_3 x_4 x_5 x_6 x_2} \\
	\textcolor{red}{x_2 x_3 x_4 x_5 x_6 x_1}
\end{cases}$ &
$ x_1 x_2 x_3 x_4 x_5 x_6 x_7 \rightarrow \begin{cases}
	\textcolor{blue}{x_2 x_3 x_1}\textcolor{red}{x_5 x_6 x_7 x_4} \\
	\textcolor{blue}{x_2 x_1}\textcolor{red}{x_4 x_5 x_6 x_7 x_3} \\
	\textcolor{blue}{x_1}\textcolor{red}{x_3 x_4 x_5 x_6 x_7 x_2} \\
	\textcolor{red}{x_2 x_3 x_4 x_5 x_6 x_7 x_1}
\end{cases}$ \\
\end{tabular}
\end{table}
Now we introduce terminology and a visual representation of these rules which we
will make use of throughout our proof. First, we note that the graph $\Gamma_n$
has $k + 1$ rules, we shall call these rules $\pi_i$ for $0 \leq i \leq k$,
where rule $\pi_i$ is given by
$\pi_i(x_1 x_2 \dots x_n) = x_2 x_3 \dots x_{k - i} x_1
	x_{k - i + 2} x_{k - i + 3} \dots x_n x_{k - i + 1}$, and
$\pi_k(x_1 x_2 \dots x_n) = x_2 x_3 \dots x_n x_1$.
In this notation, we now show our visual representation of the rules in the case
$n = 8$.
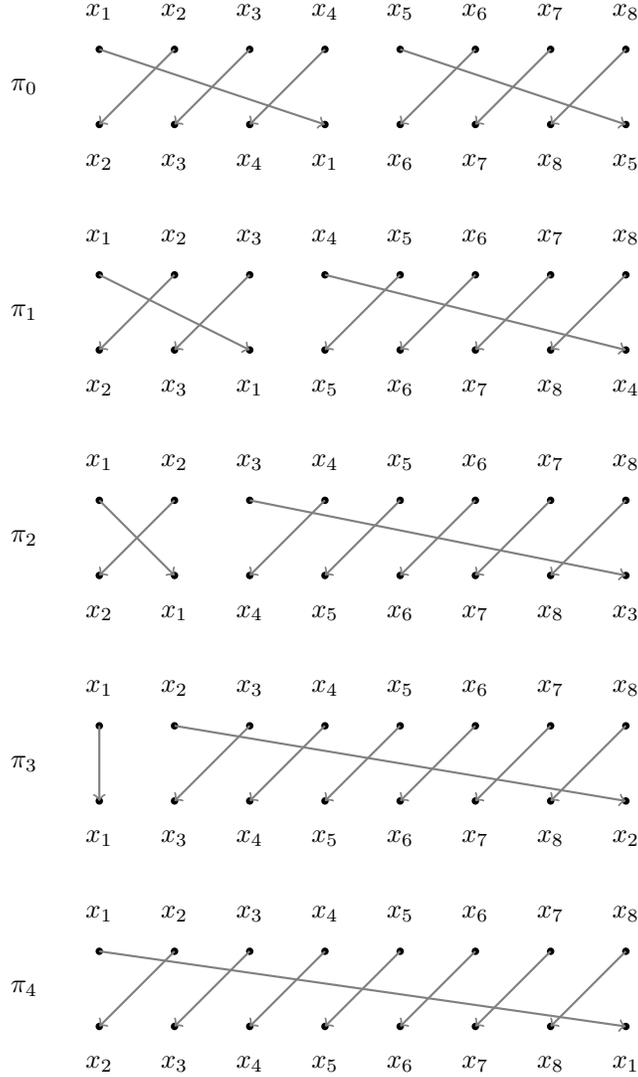
\begin{center}
\begin{tikzpicture}
	\foreach \x in {1,...,8} {
		\draw (\x,  0) node {};
		\draw (\x,  1) node {};
		\draw (\x,  3) node {};
		\draw (\x,  4) node {};
		\draw (\x,  6) node {};
		\draw (\x,  7) node {};
		\draw (\x,  9) node {};
		\draw (\x, 10) node {};
		\draw (\x, 12) node {};
		\draw (\x, 13) node {};
	}
	\draw [->, gray] (1, 13) to (4, 12);
	\draw [->, gray] (2, 13) to (1, 12);
	\draw [->, gray] (3, 13) to (2, 12);
	\draw [->, gray] (4, 13) to (3, 12);
	\draw [->, gray] (5, 13) to (8, 12);
	\draw [->, gray] (6, 13) to (5, 12);
	\draw [->, gray] (7, 13) to (6, 12);
	\draw [->, gray] (8, 13) to (7, 12);
	\foreach \x in {1,...,8} {
		\node[fill=none, draw=none] at (\x, 13.5) {$x_\x$};
	}
	\node[fill=none, draw=none] at (1, 11.5) {$x_2$};
	\node[fill=none, draw=none] at (2, 11.5) {$x_3$};
	\node[fill=none, draw=none] at (3, 11.5) {$x_4$};
	\node[fill=none, draw=none] at (4, 11.5) {$x_1$};
	\node[fill=none, draw=none] at (5, 11.5) {$x_6$};
	\node[fill=none, draw=none] at (6, 11.5) {$x_7$};
	\node[fill=none, draw=none] at (7, 11.5) {$x_8$};
	\node[fill=none, draw=none] at (8, 11.5) {$x_5$};
	\node[fill=none, draw=none] at (0, 12.5) {$\pi_0$};
	\draw [->, gray] (1, 10) to (3, 9);
	\draw [->, gray] (2, 10) to (1, 9);
	\draw [->, gray] (3, 10) to (2, 9);
	\draw [->, gray] (4, 10) to (8, 9);
	\draw [->, gray] (5, 10) to (4, 9);
	\draw [->, gray] (6, 10) to (5, 9);
	\draw [->, gray] (7, 10) to (6, 9);
	\draw [->, gray] (8, 10) to (7, 9);
	\foreach \x in {1,...,8} {
		\node[fill=none, draw=none] at (\x, 10.5) {$x_\x$};
	}
	\node[fill=none, draw=none] at (1, 8.5) {$x_2$};
	\node[fill=none, draw=none] at (2, 8.5) {$x_3$};
	\node[fill=none, draw=none] at (3, 8.5) {$x_1$};
	\node[fill=none, draw=none] at (4, 8.5) {$x_5$};
	\node[fill=none, draw=none] at (5, 8.5) {$x_6$};
	\node[fill=none, draw=none] at (6, 8.5) {$x_7$};
	\node[fill=none, draw=none] at (7, 8.5) {$x_8$};
	\node[fill=none, draw=none] at (8, 8.5) {$x_4$};
	\node[fill=none, draw=none] at (0, 9.5) {$\pi_1$};
	\draw [->, gray] (1, 7) to (2, 6);
	\draw [->, gray] (2, 7) to (1, 6);
	\draw [->, gray] (3, 7) to (8, 6);
	\draw [->, gray] (4, 7) to (3, 6);
	\draw [->, gray] (5, 7) to (4, 6);
	\draw [->, gray] (6, 7) to (5, 6);
	\draw [->, gray] (7, 7) to (6, 6);
	\draw [->, gray] (8, 7) to (7, 6);
	\foreach \x in {1,...,8} {
		\node[fill=none, draw=none] at (\x, 7.5) {$x_\x$};
	}
	\node[fill=none, draw=none] at (1, 5.5) {$x_2$};
	\node[fill=none, draw=none] at (2, 5.5) {$x_1$};
	\node[fill=none, draw=none] at (3, 5.5) {$x_4$};
	\node[fill=none, draw=none] at (4, 5.5) {$x_5$};
	\node[fill=none, draw=none] at (5, 5.5) {$x_6$};
	\node[fill=none, draw=none] at (6, 5.5) {$x_7$};
	\node[fill=none, draw=none] at (7, 5.5) {$x_8$};
	\node[fill=none, draw=none] at (8, 5.5) {$x_3$};
	\node[fill=none, draw=none] at (0, 6.5) {$\pi_2$};
	\draw [->, gray] (1, 4) to (1, 3);
	\draw [->, gray] (2, 4) to (8, 3);
	\draw [->, gray] (3, 4) to (2, 3);
	\draw [->, gray] (4, 4) to (3, 3);
	\draw [->, gray] (5, 4) to (4, 3);
	\draw [->, gray] (6, 4) to (5, 3);
	\draw [->, gray] (7, 4) to (6, 3);
	\draw [->, gray] (8, 4) to (7, 3);
	\foreach \x in {1,...,8} {
		\node[fill=none, draw=none] at (\x, 4.5) {$x_\x$};
	}
	\node[fill=none, draw=none] at (1, 2.5) {$x_1$};
	\node[fill=none, draw=none] at (2, 2.5) {$x_3$};
	\node[fill=none, draw=none] at (3, 2.5) {$x_4$};
	\node[fill=none, draw=none] at (4, 2.5) {$x_5$};
	\node[fill=none, draw=none] at (5, 2.5) {$x_6$};
	\node[fill=none, draw=none] at (6, 2.5) {$x_7$};
	\node[fill=none, draw=none] at (7, 2.5) {$x_8$};
	\node[fill=none, draw=none] at (8, 2.5) {$x_2$};
	\node[fill=none, draw=none] at (0, 3.5) {$\pi_3$};
	\draw [->, gray] (1, 1) to (8, 0);
	\draw [->, gray] (2, 1) to (1, 0);
	\draw [->, gray] (3, 1) to (2, 0);
	\draw [->, gray] (4, 1) to (3, 0);
	\draw [->, gray] (5, 1) to (4, 0);
	\draw [->, gray] (6, 1) to (5, 0);
	\draw [->, gray] (7, 1) to (6, 0);
	\draw [->, gray] (8, 1) to (7, 0);
	\foreach \x in {1,...,8} {
		\node[fill=none, draw=none] at (\x, 1.5) {$x_\x$};
	}
	\node[fill=none, draw=none] at (1, -0.5) {$x_2$};
	\node[fill=none, draw=none] at (2, -0.5) {$x_3$};
	\node[fill=none, draw=none] at (3, -0.5) {$x_4$};
	\node[fill=none, draw=none] at (4, -0.5) {$x_5$};
	\node[fill=none, draw=none] at (5, -0.5) {$x_6$};
	\node[fill=none, draw=none] at (6, -0.5) {$x_7$};
	\node[fill=none, draw=none] at (7, -0.5) {$x_8$};
	\node[fill=none, draw=none] at (8, -0.5) {$x_1$};
	\node[fill=none, draw=none] at (0,  0.5) {$\pi_4$};
\end{tikzpicture}
\end{center}
Within this visual representation, we label the following features
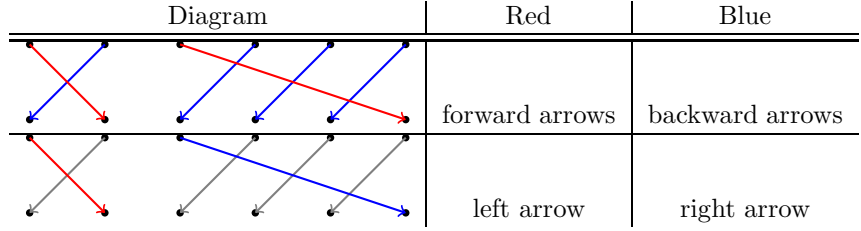
\begin{table}[H]
\centering
\begin{tabular}{c|c|c}
Diagram & Red & Blue \\ \hline \hline
\begin{tikzpicture}
	\foreach \x in {1,...,6} {
		\draw (\x, 1) node {};
		\draw (\x, 0) node {};
	}
	\draw [->, blue] (2, 1) to (1, 0);
	\draw [->, blue] (4, 1) to (3, 0);
	\draw [->, blue] (5, 1) to (4, 0);
	\draw [->, blue] (6, 1) to (5, 0);
	\draw [->, red]  (3, 1) to (6, 0);
	\draw [->, red]  (1, 1) to (2, 0);
\end{tikzpicture} & forward arrows & backward arrows \\ \hline
\begin{tikzpicture}
	\foreach \x in {1,...,6} {
		\draw (\x, 1) node {};
		\draw (\x, 0) node {};
	}
	\draw [->, gray] (2, 1) to (1, 0);
	\draw [->, gray] (4, 1) to (3, 0);
	\draw [->, gray] (5, 1) to (4, 0);
	\draw [->, gray] (6, 1) to (5, 0);
	\draw [->, blue] (3, 1) to (6, 0);
	\draw [->, red]  (1, 1) to (2, 0);
\end{tikzpicture} & left arrow & right arrow \\
\end{tabular}
\end{table}

\begin{lemma}
\label{num_arrows}
The number of right arrows in a path of length $m$ is $m$, and the number of
left arrows in a path of length $m$ is less than or equal to $m$.
\end{lemma}

\begin{proof}
Each of the rules $\pi_i$ for $0 \leq i \leq k$ contains exactly one right
arrow, and either one or zero left arrows.
\end{proof}

We represent the composition of rules as in the following diagram and call it a
\textit{path}. The following diagram shows the path $\pi_3 \pi_0 \pi_2$ when
$n = 8$.
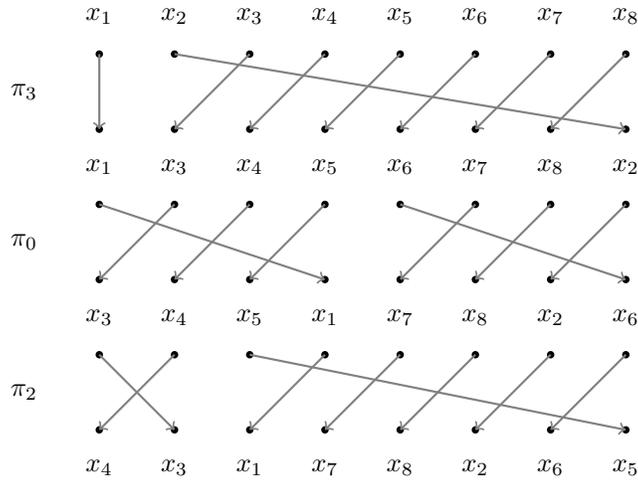
\begin{center}
\begin{tikzpicture}
	\foreach \x in {1,...,8} {
		\draw (\x, 1) node {};
		\draw (\x, 2) node {};
		\draw (\x, 3) node {};
		\draw (\x, 4) node {};
		\draw (\x, 5) node {};
		\draw (\x, 6) node {};
	}
	\draw [->, gray] (3, 6) to (2, 5);
	\draw [->, gray] (4, 6) to (3, 5);
	\draw [->, gray] (5, 6) to (4, 5);
	\draw [->, gray] (6, 6) to (5, 5);
	\draw [->, gray] (7, 6) to (6, 5);
	\draw [->, gray] (8, 6) to (7, 5);
	\draw [->, gray] (2, 6) to (8, 5);
	\draw [->, gray] (1, 6) to (1, 5);
	\draw [->, gray] (2, 4) to (1, 3);
	\draw [->, gray] (3, 4) to (2, 3);
	\draw [->, gray] (4, 4) to (3, 3);
	\draw [->, gray] (6, 4) to (5, 3);
	\draw [->, gray] (7, 4) to (6, 3);
	\draw [->, gray] (8, 4) to (7, 3);
	\draw [->, gray] (5, 4) to (8, 3);
	\draw [->, gray] (1, 4) to (4, 3);
	\draw [->, gray] (2, 2) to (1, 1);
	\draw [->, gray] (4, 2) to (3, 1);
	\draw [->, gray] (5, 2) to (4, 1);
	\draw [->, gray] (6, 2) to (5, 1);
	\draw [->, gray] (7, 2) to (6, 1);
	\draw [->, gray] (8, 2) to (7, 1);
	\draw [->, gray] (3, 2) to (8, 1);
	\draw [->, gray] (1, 2) to (2, 1);
	\foreach \x in {1,...,8} {
		\node[fill=none, draw=none] at (\x, 6.5) {$x_\x$};
	}
	\node[fill=none, draw=none] at (1, 4.5) {$x_1$};
	\node[fill=none, draw=none] at (2, 4.5) {$x_3$};
	\node[fill=none, draw=none] at (3, 4.5) {$x_4$};
	\node[fill=none, draw=none] at (4, 4.5) {$x_5$};
	\node[fill=none, draw=none] at (5, 4.5) {$x_6$};
	\node[fill=none, draw=none] at (6, 4.5) {$x_7$};
	\node[fill=none, draw=none] at (7, 4.5) {$x_8$};
	\node[fill=none, draw=none] at (8, 4.5) {$x_2$};
	\node[fill=none, draw=none] at (1, 2.5) {$x_3$};
	\node[fill=none, draw=none] at (2, 2.5) {$x_4$};
	\node[fill=none, draw=none] at (3, 2.5) {$x_5$};
	\node[fill=none, draw=none] at (4, 2.5) {$x_1$};
	\node[fill=none, draw=none] at (5, 2.5) {$x_7$};
	\node[fill=none, draw=none] at (6, 2.5) {$x_8$};
	\node[fill=none, draw=none] at (7, 2.5) {$x_2$};
	\node[fill=none, draw=none] at (8, 2.5) {$x_6$};
	\node[fill=none, draw=none] at (1, 0.5) {$x_4$};
	\node[fill=none, draw=none] at (2, 0.5) {$x_3$};
	\node[fill=none, draw=none] at (3, 0.5) {$x_1$};
	\node[fill=none, draw=none] at (4, 0.5) {$x_7$};
	\node[fill=none, draw=none] at (5, 0.5) {$x_8$};
	\node[fill=none, draw=none] at (6, 0.5) {$x_2$};
	\node[fill=none, draw=none] at (7, 0.5) {$x_6$};
	\node[fill=none, draw=none] at (8, 0.5) {$x_5$};
	\node[fill=none, draw=none] at (0, 5.5) {$\pi_3$};
	\node[fill=none, draw=none] at (0, 3.5) {$\pi_0$};
	\node[fill=none, draw=none] at (0, 1.5) {$\pi_2$};
\end{tikzpicture}
\end{center}
In subsequent diagrams, we may drop the explicit labelling of letters to present
the same path in a more succinct manner as in the example below.
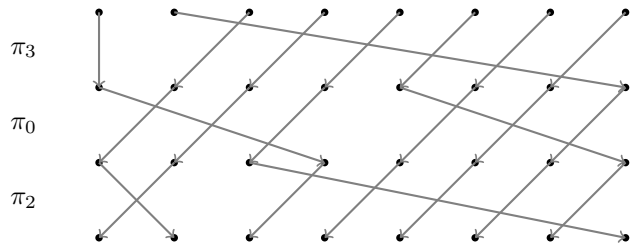
\begin{center}
\begin{tikzpicture}
	\foreach \x in {1,...,8} {
		\draw (\x, 1) node {};
		\draw (\x, 2) node {};
		\draw (\x, 3) node {};
		\draw (\x, 4) node {};
	}
	\draw [->, gray] (3, 4) to (2, 3);
	\draw [->, gray] (4, 4) to (3, 3);
	\draw [->, gray] (5, 4) to (4, 3);
	\draw [->, gray] (6, 4) to (5, 3);
	\draw [->, gray] (7, 4) to (6, 3);
	\draw [->, gray] (8, 4) to (7, 3);
	\draw [->, gray] (2, 4) to (8, 3);
	\draw [->, gray] (1, 4) to (1, 3);
	\draw [->, gray] (2, 3) to (1, 2);
	\draw [->, gray] (3, 3) to (2, 2);
	\draw [->, gray] (4, 3) to (3, 2);
	\draw [->, gray] (6, 3) to (5, 2);
	\draw [->, gray] (7, 3) to (6, 2);
	\draw [->, gray] (8, 3) to (7, 2);
	\draw [->, gray] (5, 3) to (8, 2);
	\draw [->, gray] (1, 3) to (4, 2);
	\draw [->, gray] (2, 2) to (1, 1);
	\draw [->, gray] (4, 2) to (3, 1);
	\draw [->, gray] (5, 2) to (4, 1);
	\draw [->, gray] (6, 2) to (5, 1);
	\draw [->, gray] (7, 2) to (6, 1);
	\draw [->, gray] (8, 2) to (7, 1);
	\draw [->, gray] (3, 2) to (8, 1);
	\draw [->, gray] (1, 2) to (2, 1);
	\node[fill=none, draw=none] at (0, 3.5) {$\pi_3$};
	\node[fill=none, draw=none] at (0, 2.5) {$\pi_0$};
	\node[fill=none, draw=none] at (0, 1.5) {$\pi_2$};
\end{tikzpicture}
\end{center}
We will refer to the \textit{trail} from position $i$ in a path to mean the
concatenation of consecutive arrows in our diagram starting from the arrow at
position $i$. In the following example we have highlighted the trail starting
at position 2.
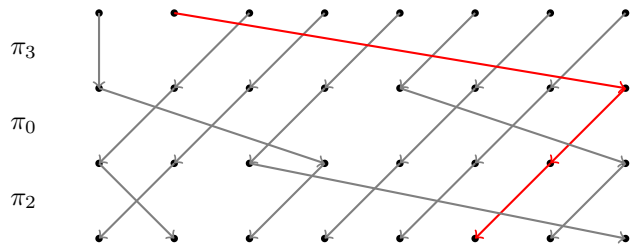
\begin{center}
\begin{tikzpicture}
	\foreach \x in {1,...,8} {
		\draw (\x, 1) node {};
		\draw (\x, 2) node {};
		\draw (\x, 3) node {};
		\draw (\x, 4) node {};
	}
	\draw [->, gray] (3, 4) to (2, 3);
	\draw [->, gray] (4, 4) to (3, 3);
	\draw [->, gray] (5, 4) to (4, 3);
	\draw [->, gray] (6, 4) to (5, 3);
	\draw [->, gray] (7, 4) to (6, 3);
	\draw [->, gray] (8, 4) to (7, 3);
	\draw [->, red]  (2, 4) to (8, 3);
	\draw [->, gray] (1, 4) to (1, 3);
	\draw [->, gray] (2, 3) to (1, 2);
	\draw [->, gray] (3, 3) to (2, 2);
	\draw [->, gray] (4, 3) to (3, 2);
	\draw [->, gray] (6, 3) to (5, 2);
	\draw [->, gray] (7, 3) to (6, 2);
	\draw [->, red]  (8, 3) to (7, 2);
	\draw [->, gray] (5, 3) to (8, 2);
	\draw [->, gray] (1, 3) to (4, 2);
	\draw [->, gray] (2, 2) to (1, 1);
	\draw [->, gray] (4, 2) to (3, 1);
	\draw [->, gray] (5, 2) to (4, 1);
	\draw [->, gray] (6, 2) to (5, 1);
	\draw [->, red]  (7, 2) to (6, 1);
	\draw [->, gray] (8, 2) to (7, 1);
	\draw [->, gray] (3, 2) to (8, 1);
	\draw [->, gray] (1, 2) to (2, 1);
	\node[fill=none, draw=none] at (0, 3.5) {$\pi_3$};
	\node[fill=none, draw=none] at (0, 2.5) {$\pi_0$};
	\node[fill=none, draw=none] at (0, 1.5) {$\pi_2$};
\end{tikzpicture}
\end{center}
We will call a trail \textit{closed} if it begins and ends at the same position.
Here we illustrate a closed trail in blue, and a non-closed trail in red.
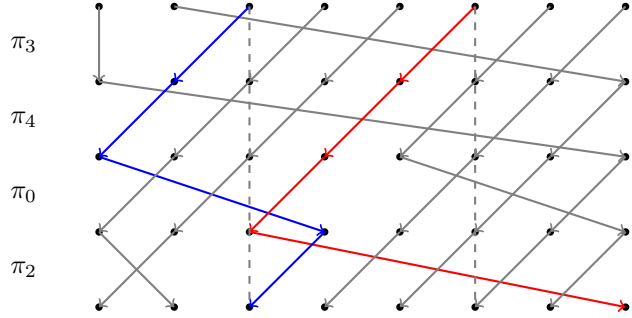
\begin{center}
\begin{tikzpicture}
	\foreach \x in {1,...,8} {
		\draw (\x, 1) node {};
		\draw (\x, 2) node {};
		\draw (\x, 3) node {};
		\draw (\x, 4) node {};
		\draw (\x, 5) node {};
	}
	\draw [->, gray] (1, 5) to (1, 4);
	\draw [->, gray] (2, 5) to (8, 4);
	\draw [->, blue] (3, 5) to (2, 4);
	\draw [->, gray] (4, 5) to (3, 4);
	\draw [->, gray] (5, 5) to (4, 4);
	\draw [->,  red] (6, 5) to (5, 4);
	\draw [->, gray] (7, 5) to (6, 4);
	\draw [->, gray] (8, 5) to (7, 4);
	\draw [->, gray] (1, 4) to (8, 3);
	\draw [->, blue] (2, 4) to (1, 3);
	\draw [->, gray] (3, 4) to (2, 3);
	\draw [->, gray] (4, 4) to (3, 3);
	\draw [->,  red] (5, 4) to (4, 3);
	\draw [->, gray] (6, 4) to (5, 3);
	\draw [->, gray] (7, 4) to (6, 3);
	\draw [->, gray] (8, 4) to (7, 3);
	\draw [->, blue] (1, 3) to (4, 2);
	\draw [->, gray] (2, 3) to (1, 2);
	\draw [->, gray] (3, 3) to (2, 2);
	\draw [->,  red] (4, 3) to (3, 2);
	\draw [->, gray] (5, 3) to (8, 2);
	\draw [->, gray] (6, 3) to (5, 2);
	\draw [->, gray] (7, 3) to (6, 2);
	\draw [->, gray] (8, 3) to (7, 2);
	\draw [->, gray] (1, 2) to (2, 1);
	\draw [->, gray] (2, 2) to (1, 1);
	\draw [->,  red] (3, 2) to (8, 1);
	\draw [->, blue] (4, 2) to (3, 1);
	\draw [->, gray] (5, 2) to (4, 1);
	\draw [->, gray] (6, 2) to (5, 1);
	\draw [->, gray] (7, 2) to (6, 1);
	\draw [->, gray] (8, 2) to (7, 1);
	\draw [-, dashed, gray] (3, 5) to (3, 1);
	\draw [-, dashed, gray] (6, 5) to (6, 1);
	\node[fill=none, draw=none] at (0, 4.5) {$\pi_3$};
	\node[fill=none, draw=none] at (0, 3.5) {$\pi_4$};
	\node[fill=none, draw=none] at (0, 2.5) {$\pi_0$};
	\node[fill=none, draw=none] at (0, 1.5) {$\pi_2$};
\end{tikzpicture}
\end{center}
We will call a path \textit{closed} if the trails starting at each position in
the path are closed.

With the terminology established, we now briefly describe our motivation.
In order to prove the G\'{o}mez graphs are subregular and alphabet stable, we
aim to count paths of lengths $n$ and $n - 1$ from each
neighbour of an arbitrary vertex $v \in \Gamma_n$ back to $v$. All of these
paths in $\Gamma_n$ correspond to cycles of lengths $n$ and $n + 1$ in
$\Gamma_n$, which correspond exactly to the closed paths we have just defined.
Hence, we now aim to count closed paths of lengths $n$ and $n + 1$,
considering what the first rule is on those paths.

For a path $p = p_0 p_1 \dots p_n$ we shall call
$p^i = p_{i + 1} p_{i + 2} \dots p_n p_1 \dots p_i$ the $i$\textsuperscript{th}
\textit{rotation} of $p$.

\begin{lemma}
\label{rotate}
There is a bijection between the closed trails of a path $p$ and the closed
paths of each of its rotations $p^i$.
\end{lemma}

\begin{proof}
We demonstrate such a bijection between the closed trails of some path $p$ and
its rotation $p^2$, noting that the result follows from a trivial induction.
Let $p$ be a path, first we show that the trail at $i$ is closed in $p$ if, and
only if, the trail at $p_1(i)$ is closed in $p^2$.
\begin{align*}
	p(i) = i \quad &\Leftrightarrow \quad
			(p_2 p_3 \dots p_n)(p_1(i)) = i \\
		&\Leftrightarrow \quad
			(p_2 p_3 \dots p_n p_1)(p_1(i)) = p_1(i)
			\quad \Leftrightarrow \quad
			p^2(p_1(i)) = p_1(i).
\end{align*}
Hence, as $p_1$ is a bijection, we have a bijection between the closed trails of
$p$ and $p^2$.
\end{proof}

In light of Lemma~\ref{rotate}, we shall identify a closed trail starting at $i$
in a path $p$ with the closed trail starting at $(p_1 p_2 \dots p_{j - 1})(i)$
in $p^j$, referring to them as the same trail.
Here we illustrate an example of a closed trail and its rotations.
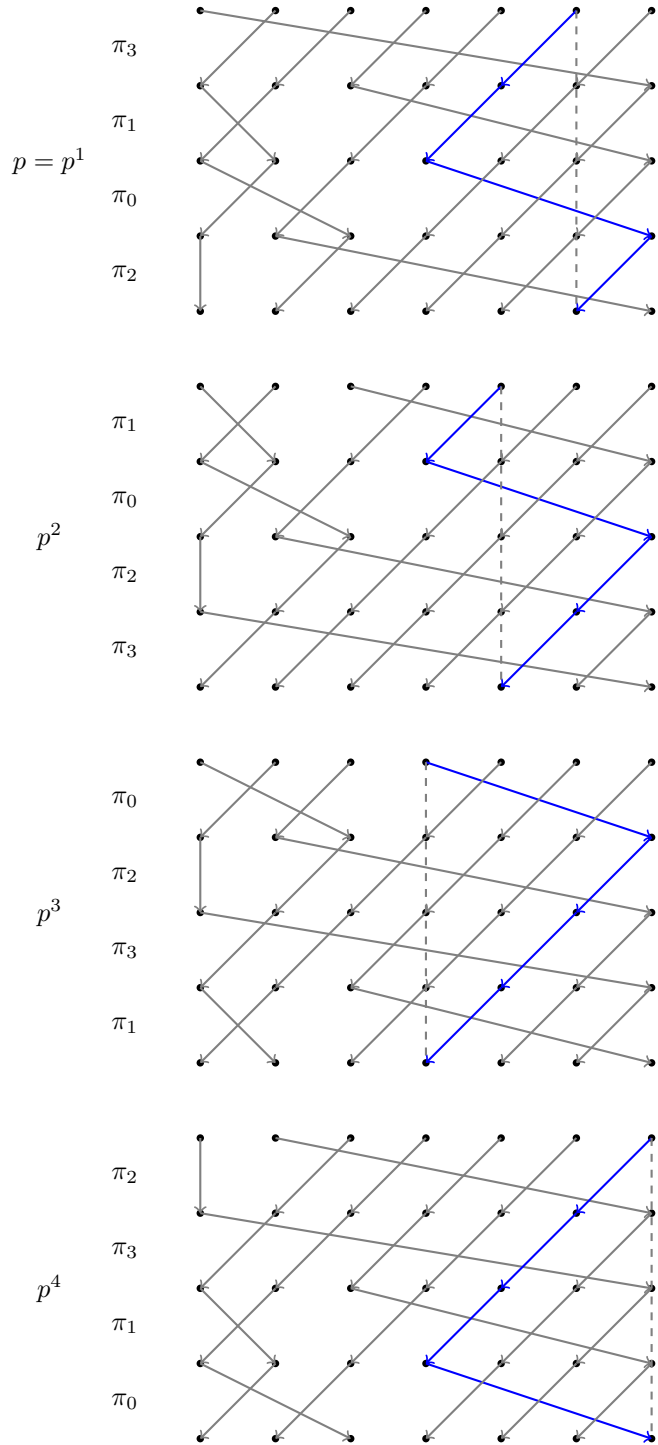
\begin{center}
\begin{tikzpicture}
	\foreach \x in {1,...,7} {
		\draw (\x, 1) node {};
		\draw (\x, 2) node {};
		\draw (\x, 3) node {};
		\draw (\x, 4) node {};
		\draw (\x, 5) node {};
		\draw (\x,  6) node {};
		\draw (\x,  7) node {};
		\draw (\x,  8) node {};
		\draw (\x,  9) node {};
		\draw (\x, 10) node {};
		\draw (\x, 11) node {};
		\draw (\x, 12) node {};
		\draw (\x, 13) node {};
		\draw (\x, 14) node {};
		\draw (\x, 15) node {};
		\draw (\x, 16) node {};
		\draw (\x, 17) node {};
		\draw (\x, 18) node {};
		\draw (\x, 19) node {};
		\draw (\x, 20) node {};
	}
	\draw [->, gray] (1, 20) to (7, 19);
	\draw [->, gray] (2, 20) to (1, 19);
	\draw [->, gray] (3, 20) to (2, 19);
	\draw [->, gray] (4, 20) to (3, 19);
	\draw [->, gray] (5, 20) to (4, 19);
	\draw [->, blue] (6, 20) to (5, 19);
	\draw [->, gray] (7, 20) to (6, 19);
	\draw [->, gray] (1, 19) to (2, 18);
	\draw [->, gray] (2, 19) to (1, 18);
	\draw [->, gray] (3, 19) to (7, 18);
	\draw [->, gray] (4, 19) to (3, 18);
	\draw [->, blue] (5, 19) to (4, 18);
	\draw [->, gray] (6, 19) to (5, 18);
	\draw [->, gray] (7, 19) to (6, 18);
	\draw [->, gray] (1, 18) to (3, 17);
	\draw [->, gray] (2, 18) to (1, 17);
	\draw [->, gray] (3, 18) to (2, 17);
	\draw [->, blue] (4, 18) to (7, 17);
	\draw [->, gray] (5, 18) to (4, 17);
	\draw [->, gray] (6, 18) to (5, 17);
	\draw [->, gray] (7, 18) to (6, 17);
	\draw [->, gray] (1, 17) to (1, 16);
	\draw [->, gray] (2, 17) to (7, 16);
	\draw [->, gray] (3, 17) to (2, 16);
	\draw [->, gray] (4, 17) to (3, 16);
	\draw [->, gray] (5, 17) to (4, 16);
	\draw [->, gray] (6, 17) to (5, 16);
	\draw [->, blue] (7, 17) to (6, 16);
	\node[fill=none, draw=none] at (0, 19.5) {$\pi_3$};
	\node[fill=none, draw=none] at (0, 18.5) {$\pi_1$};
	\node[fill=none, draw=none] at (0, 17.5) {$\pi_0$};
	\node[fill=none, draw=none] at (0, 16.5) {$\pi_2$};
	\node[fill=none, draw=none] at (-1, 18) {$p = p^1$};
	\draw [-, dashed, gray] (6, 20) to (6, 16);
	\draw [->, gray] (1, 15) to (2, 14);
	\draw [->, gray] (2, 15) to (1, 14);
	\draw [->, gray] (3, 15) to (7, 14);
	\draw [->, gray] (4, 15) to (3, 14);
	\draw [->, blue] (5, 15) to (4, 14);
	\draw [->, gray] (6, 15) to (5, 14);
	\draw [->, gray] (7, 15) to (6, 14);
	\draw [->, gray] (1, 14) to (3, 13);
	\draw [->, gray] (2, 14) to (1, 13);
	\draw [->, gray] (3, 14) to (2, 13);
	\draw [->, blue] (4, 14) to (7, 13);
	\draw [->, gray] (5, 14) to (4, 13);
	\draw [->, gray] (6, 14) to (5, 13);
	\draw [->, gray] (7, 14) to (6, 13);
	\draw [->, gray] (1, 13) to (1, 12);
	\draw [->, gray] (2, 13) to (7, 12);
	\draw [->, gray] (3, 13) to (2, 12);
	\draw [->, gray] (4, 13) to (3, 12);
	\draw [->, gray] (5, 13) to (4, 12);
	\draw [->, gray] (6, 13) to (5, 12);
	\draw [->, blue] (7, 13) to (6, 12);
	\draw [->, gray] (1, 12) to (7, 11);
	\draw [->, gray] (2, 12) to (1, 11);
	\draw [->, gray] (3, 12) to (2, 11);
	\draw [->, gray] (4, 12) to (3, 11);
	\draw [->, gray] (5, 12) to (4, 11);
	\draw [->, blue] (6, 12) to (5, 11);
	\draw [->, gray] (7, 12) to (6, 11);
	\node[fill=none, draw=none] at (0, 14.5) {$\pi_1$};
	\node[fill=none, draw=none] at (0, 13.5) {$\pi_0$};
	\node[fill=none, draw=none] at (0, 12.5) {$\pi_2$};
	\node[fill=none, draw=none] at (0, 11.5) {$\pi_3$};
	\node[fill=none, draw=none] at (-1, 13) {$p^2$};
	\draw [-, dashed, gray] (5, 15) to (5, 11);
	\draw [->, gray] (1, 10) to (3, 9);
	\draw [->, gray] (2, 10) to (1, 9);
	\draw [->, gray] (3, 10) to (2, 9);
	\draw [->, blue] (4, 10) to (7, 9);
	\draw [->, gray] (5, 10) to (4, 9);
	\draw [->, gray] (6, 10) to (5, 9);
	\draw [->, gray] (7, 10) to (6, 9);
	\draw [->, gray] (1, 9) to (1, 8);
	\draw [->, gray] (2, 9) to (7, 8);
	\draw [->, gray] (3, 9) to (2, 8);
	\draw [->, gray] (4, 9) to (3, 8);
	\draw [->, gray] (5, 9) to (4, 8);
	\draw [->, gray] (6, 9) to (5, 8);
	\draw [->, blue] (7, 9) to (6, 8);
	\draw [->, gray] (1, 8) to (7, 7);
	\draw [->, gray] (2, 8) to (1, 7);
	\draw [->, gray] (3, 8) to (2, 7);
	\draw [->, gray] (4, 8) to (3, 7);
	\draw [->, gray] (5, 8) to (4, 7);
	\draw [->, blue] (6, 8) to (5, 7);
	\draw [->, gray] (7, 8) to (6, 7);
	\draw [->, gray] (1, 7) to (2, 6);
	\draw [->, gray] (2, 7) to (1, 6);
	\draw [->, gray] (3, 7) to (7, 6);
	\draw [->, gray] (4, 7) to (3, 6);
	\draw [->, blue] (5, 7) to (4, 6);
	\draw [->, gray] (6, 7) to (5, 6);
	\draw [->, gray] (7, 7) to (6, 6);
	\node[fill=none, draw=none] at (0, 9.5) {$\pi_0$};
	\node[fill=none, draw=none] at (0, 8.5) {$\pi_2$};
	\node[fill=none, draw=none] at (0, 7.5) {$\pi_3$};
	\node[fill=none, draw=none] at (0, 6.5) {$\pi_1$};
	\node[fill=none, draw=none] at (-1, 8) {$p^3$};
	\draw [-, dashed, gray] (4, 10) to (4, 6);
	\draw [->, gray] (1, 5) to (1, 4);
	\draw [->, gray] (2, 5) to (7, 4);
	\draw [->, gray] (3, 5) to (2, 4);
	\draw [->, gray] (4, 5) to (3, 4);
	\draw [->, gray] (5, 5) to (4, 4);
	\draw [->, gray] (6, 5) to (5, 4);
	\draw [->, blue] (7, 5) to (6, 4);
	\draw [->, gray] (1, 4) to (7, 3);
	\draw [->, gray] (2, 4) to (1, 3);
	\draw [->, gray] (3, 4) to (2, 3);
	\draw [->, gray] (4, 4) to (3, 3);
	\draw [->, gray] (5, 4) to (4, 3);
	\draw [->, blue] (6, 4) to (5, 3);
	\draw [->, gray] (7, 4) to (6, 3);
	\draw [->, gray] (1, 3) to (2, 2);
	\draw [->, gray] (2, 3) to (1, 2);
	\draw [->, gray] (3, 3) to (7, 2);
	\draw [->, gray] (4, 3) to (3, 2);
	\draw [->, blue] (5, 3) to (4, 2);
	\draw [->, gray] (6, 3) to (5, 2);
	\draw [->, gray] (7, 3) to (6, 2);
	\draw [->, gray] (1, 2) to (3, 1);
	\draw [->, gray] (2, 2) to (1, 1);
	\draw [->, gray] (3, 2) to (2, 1);
	\draw [->, blue] (4, 2) to (7, 1);
	\draw [->, gray] (5, 2) to (4, 1);
	\draw [->, gray] (6, 2) to (5, 1);
	\draw [->, gray] (7, 2) to (6, 1);
	\node[fill=none, draw=none] at (0, 4.5) {$\pi_2$};
	\node[fill=none, draw=none] at (0, 3.5) {$\pi_3$};
	\node[fill=none, draw=none] at (0, 2.5) {$\pi_1$};
	\node[fill=none, draw=none] at (0, 1.5) {$\pi_0$};
	\node[fill=none, draw=none] at (-1, 3) {$p^4$};
	\draw [-, dashed, gray] (7, 5) to (7, 1);
\end{tikzpicture}
\end{center}

\begin{corollary}
A path $p$ is closed if and only if each of its rotations $p^i$ is closed.
\end{corollary}

\begin{lemma}
\label{at_least_two}
Any closed trail in a path of length $n + 1$ must contain at least two
forward arrows.
\end{lemma}

\begin{proof}
For any closed trail we see that the
distance traversed by backwards arrows is equal to the distance traversed by
forwards arrows. If there are no forwards arrows, this obviously cannot happen.
If there is only one forwards arrow, then there must be $n$ backwards arrows,
hence the forwards arrow must correspond to travelling forwards $n$ places.
However, the furthest that can be travelled forwards occurs in rule $\pi_k$,
which travels forwards by $n - 1$ spaces. Hence, this cannot occur, and so any
closed trail must contain at least two forward
arrows.
\end{proof}

\begin{lemma}
\label{left_arrows_suck}
Any closed trail of length $n + 1$ whose only forward arrows are left
arrows contains at least three left arrows.
\end{lemma}

\begin{proof}
The furthest that can be travelled forwards by a left arrow occurs in the rule
$\pi_0$ in which we travel forwards $k - 1$ spaces. If we have a closed trail
which contains two left arrows, then it travels a
distance of $n - 1$ with backwards arrows. Hence, we must have
$n - 1 \leq 2(k - 1) = 2k - 2 \leq n - 2$.
\end{proof}

\begin{lemma}
\label{right_arrows_rule}
Any closed trail of length $n + 1$ whose only forward arrows are right arrows
contains exactly two right arrows.
\end{lemma}

\begin{proof}
Consider a closed trail whose only forward arrows are right arrows. As the trail
is closed, the sum of the forward arrows equals the sum of the backward arrows.
Suppose there are three or more right arrows in the
trail, then the sum of the forward arrows is at least $3(k - 1)$ and the sum of
the backwards arrows is at most $n - 3 < 3(k - 1)$, hence there must be fewer
than three right arrows, and so by Lemma~\ref{at_least_two} there are two.
\end{proof}

\begin{lemma}
\label{at_most_three}
In a closed path of length $n + 1$ there are at most
three trails containing two right arrows.
\end{lemma}

\begin{proof}
Suppose we have a closed path of length $n + 1$ which contains
at least four trails containing two right arrows.
By Lemma~\ref{num_arrows} we
now have $(n + 1) - 8 = n - 7$ unaccounted for right arrows,
and at most $n + 1$
unaccounted for left arrows. Lemma~\ref{at_least_two} tells us that we need to
use at least two forward arrows per remaining trail in the path.
If each left arrow is in a trail with a
right arrow, then such a trail requires only two arrows, otherwise it requires
three or more. Hence, in order to minimise the number of required arrows, we may
assume as many left arrows as possible are paired with right arrows. In this
manner, we assume all $n - 7$ remaining right arrows are paired with left
arrows. This leaves at most $(n + 1) - (n - 7) = 8$ unaccounted for left arrows.
We have now accounted for $(n - 7) + 4 = n - 3$ of $n$ trails, leaving three
unaccounted for.
Lemma~\ref{left_arrows_suck} tells us that each of the three remaining trails
requires at least three left arrows, but we only have eight unaccounted for left
arrows.
\end{proof}

\begin{lemma}
\label{right_arrow_first}
In a path of length $n + 1$, if the trail starting with the right arrow of $p_1$
contains no further right arrows, it contains the left arrow of $p_{n + 1}$.
\end{lemma}

\begin{proof}
After $p_1$, the trail is at position $n$. As the trail contains no further
right arrows, each $p_{1 + i}$ maps the trail from
$n - i$ to $n - i - 1$, provided that $n - i > 1$.
Hence, the trail reaches position 1 after $p_n$, and so the trail
contains the left arrow of $p_{n + 1}$.
\end{proof}

If in a path $p$ of length $n$ there is some $i$ and $j$ such that $p_i = \pi_j$
and $p_{i + 1} = \pi_{j + 1}$, then we call the left arrow of $p_i$ and the
right arrow of $p_{i + 1}$ \textit{paired} and refer to them together as a
\textit{pair}. As a special case, we allow $i = n$ and use $p_1$ instead of
$p_{i + 1}$.

\begin{lemma}
\label{right_left_pair}
If a closed trail in a path of length $n + 1$ contains both right and left
arrows, then it contains one pair and no other forward arrows.
\end{lemma}

\begin{proof}
Suppose $p$ is a path with such a trail. Let $q$ be a rotation of $p$ which puts
a right arrow of the trail in position $q_1$. As the trail contains a left
arrow, at some point we have some $i$ such that $(q_2 q_3 \dots q_i)(n) = 1$.
As the most we can
travel backwards in each $q_j$ is one space, the soonest this can happen is by
$i = n - 1$, if and only if each $q_j$ introduces a backwards arrow into the
trail. If this does not happen, then no $q_j$ including $j \in \{ 1, n + 1 \}$
can contain a left arrow in the trail. Hence, this is the only possibity. If the
position the trail starts at is $k - \alpha$ then we know $q_1(k - \alpha) = n$,
so $q_1 = \pi_\alpha$, and $q_{n + 1}(1) = k - \alpha$, hence
$q_{n + 1} = \pi_{\alpha - 1}$. Hence the trail contains one pair and all other
arrows are backwards arrows.
\end{proof}

\begin{lemma}
\label{left_right_ident}
If all right arrows in a path $p$ of length $n + 1$ are either in closed trails
or paired, then all pairs in $p$ are in distinct closed trails.
\end{lemma}

\begin{proof}
Consider an arbitrary right arrow in $p$, and the rotation of $p$ which brings
this right arrow to $p_1$. If the trail from this right arrow enters a forwards
arrow before $p_{n + 1}$, the forwards arrow must be an unpaired right arrow,
and so the trail is closed. Otherwise, the trail enters a left arrow at
$p_{n + 1}$, which must be the pair of the right arrow of $p_1$, and the trail
is closed.
\end{proof}

We now note from what we have shown that if a closed trail contains a right arrow
then either it contains exactly two right arrows or it contains one pair. Hence,
we are in a position to easily deal with closed trails containing at least one
right arrow. In order to settle the case of closed trails comprised entirely of
left arrows, we now require a further definition to continue our discussion. We
shall say that within a permutation in a path, a trail is on the
\textit{left side} if it is in a cycle containing a left arrow, and the
\textit{right side} otherwise. We shall say that between two rules a trail
changes sides from right to left or from left to right between two permutations
in the obvious manner. Below is a diagram to clarify the terminology, with
arrows on the left drawn in red and arrows on the right drawn in blue.

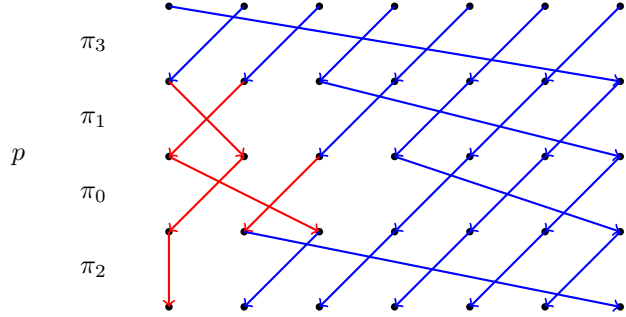
\begin{center}
\begin{tikzpicture}
	\foreach \x in {1,...,7} {
		\draw (\x, 1) node {};
		\draw (\x, 2) node {};
		\draw (\x, 3) node {};
		\draw (\x, 4) node {};
		\draw (\x, 5) node {};
	}
	\draw [->, blue] (1, 5) to (7, 4);
	\draw [->, blue] (2, 5) to (1, 4);
	\draw [->, blue] (3, 5) to (2, 4);
	\draw [->, blue] (4, 5) to (3, 4);
	\draw [->, blue] (5, 5) to (4, 4);
	\draw [->, blue] (6, 5) to (5, 4);
	\draw [->, blue] (7, 5) to (6, 4);
	\draw [->,  red] (1, 4) to (2, 3);
	\draw [->,  red] (2, 4) to (1, 3);
	\draw [->, blue] (3, 4) to (7, 3);
	\draw [->, blue] (4, 4) to (3, 3);
	\draw [->, blue] (5, 4) to (4, 3);
	\draw [->, blue] (6, 4) to (5, 3);
	\draw [->, blue] (7, 4) to (6, 3);
	\draw [->,  red] (1, 3) to (3, 2);
	\draw [->,  red] (2, 3) to (1, 2);
	\draw [->,  red] (3, 3) to (2, 2);
	\draw [->, blue] (4, 3) to (7, 2);
	\draw [->, blue] (5, 3) to (4, 2);
	\draw [->, blue] (6, 3) to (5, 2);
	\draw [->, blue] (7, 3) to (6, 2);
	\draw [->,  red] (1, 2) to (1, 1);
	\draw [->, blue] (2, 2) to (7, 1);
	\draw [->, blue] (3, 2) to (2, 1);
	\draw [->, blue] (4, 2) to (3, 1);
	\draw [->, blue] (5, 2) to (4, 1);
	\draw [->, blue] (6, 2) to (5, 1);
	\draw [->, blue] (7, 2) to (6, 1);
	\node[fill=none, draw=none] at (0, 4.5) {$\pi_3$};
	\node[fill=none, draw=none] at (0, 3.5) {$\pi_1$};
	\node[fill=none, draw=none] at (0, 2.5) {$\pi_0$};
	\node[fill=none, draw=none] at (0, 1.5) {$\pi_2$};
	\node[fill=none, draw=none] at (-1, 3) {$p$};
\end{tikzpicture}
\end{center}

For the remainder of this section, we will consider paths with the property
that, letting $p$ be our path,
$p_i = \pi_j$ and $p_{i + 1} = \pi_k$ implies $j \geq k - 1$ (note that this
restriction applies to the last entry of $p$, considering $p_1$ instead of
$p_{i + 1}$).

\begin{lemma}
If $p$ is closed, any trail which changes sides contains a pair.
\end{lemma}

\begin{proof}
For any closed trail, the number of times the trail changes from left to right
must equal the number of times the trail changes from right to left. Hence, if a
trail changes sides at all we know at some point the trail must change sides
from left to right. It is only possible for a trail to change sides from left to
right between a pair of consecutive rules of the form $p_i = \pi_j$ and
$p_{i + 1} = \pi_k$ where $k > j$. Given our assumption on our path $p$, we see
this happens only if $k = j + 1$, in which case there is only one path which
changes sides from left to right, corresponding to the left arrow of $p_i$
connecting to the right arrow of $p_{i + 1}$. Hence any trail which changes
sides must contain a pair.
\end{proof}

\begin{corollary}
\label{always_left}
If $p$ is closed, any trail which contains only left arrows
is always on the left side.
\end{corollary}

Now we shall define the \textit{closure} of a path $p$ to be $p$ concatenated
with itself the smallest number of times necessary to form a closed path.
As $p$ is a permutation, we know that the closure exists as each permutation
has finite order.

\begin{lemma}
If $p$ is a path in which every trail with a right arrow is closed, all trails
with only left arrows are always on the left side.
\end{lemma}

\begin{proof}
Letting $q$ be the closure of $p$ we may use Corollary~\ref{always_left} to see
that all trails of $q$ with only left arrows are always on the left side. The
closed trails of $q$ which contain only left arrows correspond to the
trails of $p$ which contain only left arrows. This is because any trail
containing a right arrow in $q$ corresponds to one of the closed trails of $p$.
Now, we note that the property of being on the left or right side in a path only
depends on the position of the trail at each $i$\textsuperscript{th} rule in the
path, which are the same in $p$ and $q$.
\end{proof}

\begin{lemma}
If $p$ is a path with all trails containing right arrows closed, and the trails
starting at positions $a_1, a_2, \dots, a_k$ are all the trails containing
only left arrows, and there are $m$ left arrows in total in these trails, then
$p$ maps $a_i$ to $a_{i - m}$ (subscripts considered modular).
\end{lemma}

\begin{proof}
This is provable by a trivial induction. Firstly, as all trails other than those
starting at each $a_i$ are closed, we see that $p$ maps each $a_i$ to some
$a_j$. Now, as all the trails are always on the left side, only two things may
happen at each rule. Either all trails are mapped backward by backward arrows,
and thus their left right ordering is preserved, or the trail on the far left is
mapped by a left arrow and becomes the trail on the far right. The latter case
happens exactly $m$ times. Hence, the left right ordering of the trails starting
at $a_1, a_2, \dots, a_k$ is cycled $m$ times.
\end{proof}

\begin{corollary}
\label{theyre_a_transposition}
If a path $p$ has all trails with right arrows closed, and contains exactly two
trails whose only forward arrows are left arrows, and those trails together
contain an even number of left arrows, then the path $p$ is closed.
\end{corollary}

\begin{proof}
Letting the trails starting at positions $a_1$ and $a_2$ be those containing
only left forward arrows, we may apply the previous lemma to show that $a_1$ is
mapped to $a_{1 - 6} = a_1$ and $a_2$ is mapped to $a_{2 - 6} = a_2$, hence the
trails at $a_1$ and $a_2$ are closed.
\end{proof}

\section{The Odd Case}

We now begin to count the closed paths of length $n + 1$
in the case where $n = 2k + 1$.
Throughtout this section, we consider the path $p = p_1 p_2 \dots p_{n + 1}$
which is assumed to be closed.

\begin{lemma}
\label{mirror_zeroes}
If $p_1 = \pi_0$, then $p_{k + 1} = \pi_0$, and the
trail beginning at $k + 1$ contains two right arrows.
\end{lemma}

\begin{proof}
We consider the trail starting with the right arrow of $p_1$. As this trail is
closed, by Lemma~\ref{right_arrows_rule} and Lemma~\ref{right_left_pair} we see
that it contains one more forward arrow which is either a right or a left arrow.
Hence, this trail maps backwards $n - 1$ spaces, the right arrow at $p_1$ maps
forward $k$ spaces, and therefore the other forwards arrow maps forward
$(n - 1) - k = k$ spaces. The most any left arrow maps forward is $k - 1$
spaces, hence the other arrow in the trail is a right arrow. The only rule with
a right arrow which maps forward $k$ spaces is $\pi_0$. To see
$p_{k + 1} = \pi_0$, we simply follow the backwards arrows after $p_1$.
\end{proof}

\begin{corollary}
\label{at_most_six}
There are at most six occurences of the rule $\pi_0$ in the path $p$.
\end{corollary}

\begin{proof}
This is the result of the combination of Lemma~\ref{mirror_zeroes} and
Lemma~\ref{at_most_three}.
\end{proof}

\begin{lemma}
\label{decreasing}
If $p_1 = \pi_i$ for some $i \geq 1$, then $p_{n + 1} = \pi_{i - 1}$.
\end{lemma}

\begin{proof}
Consider the trail starting with the right arrow of $p_1$. If this trail
contains a left arrow we apply Lemma~\ref{right_left_pair} and are done.
Otherwise, Lemma~\ref{right_arrows_rule} shows that
there is exactly one other right arrow in the trail. The distance mapped
backward in the trail is $n - 1$, and the distance mapped forward by the
right arrow in $p_1$ is $k + i$, hence the distance mapped forward by the other
right arrow is $(n - 1) - (k + i) = k - i$, but all right arrows map forward at
least $k$.
\end{proof}

\begin{corollary}
\label{useful_decreasing}
If $p_i = \pi_j$ for $j \neq 0$, then $p_{i - 1} = \pi_{j - 1}$.
\end{corollary}

\begin{proposition}
\label{theyre_tau_sequences}
The path $p$ is closed if, and only if,
\[
	p_i = \begin{cases}
		\pi_{a_i} & \text{for $1 \leq i \leq k + 1$}, \\
		\pi_{a_j} & \text{for $i = j + (k + 1)$, $1 \leq j \leq k + 1$},
	\end{cases}
\]
where $a_1, a_2, \dots, a_{k + 1}$ is a $\tau$-sequence.
\end{proposition}

\begin{proof}
To show the implication, we first note
the combination of Lemma~\ref{mirror_zeroes} and
Corollary~\ref{useful_decreasing} show that $p_i = p_{i + (k + 1)}$.
Hence, we may now define the sequence $a_1 a_2 \dots a_{k + 1}$ such that
$p_i = \pi_{a_i}$, and therefore $p_{i + (k + 1)} = \pi_{a_i}$.
By Corollary~\ref{at_most_six}, we see that 0 may occur at most three times in
$\langle a_i \rangle$.
Finally, again from applying Corollary~\ref{useful_decreasing} and
considering rotations of $p$, we see that $\langle a_i \rangle$
is a $\tau$-sequence.

Now, to show the reverse implication, let $\langle a_i \rangle$ be an arbitrary
$\tau$-sequence and define the corresponding path $p$. As shown in
Lemma~\ref{mirror_zeroes} if some $a_i = 0$, this corresponds to a closed trail
with the two right arrows in $p_i$ and $p_{i + (k + 1)}$. Otherwise, if
$a_i \neq 0$, then $a_{i - 1} = a_i - 1$, and the right arrow of $p_i$ is paired
with the left arrow of $p_{i - 1}$. Therefore, all right arrows of $p$ are
either in closed trails with right arrows only or are paired, hence by
Lemma~\ref{left_right_ident} each pair corresponds to a distinct closed trail.

Again by Lemma~\ref{mirror_zeroes}, each pair of zeroes introduces exactly one
closed trail. In a $\tau$-sequence, we will have either one, two or three
zeroes, and hence one, two or three pairs
of zeroes in our path. We now consider each case separately, and count the
number of closed trails containing right arrows.

\begin{itemize}
	\item If we have one pair of zeroes in our path, these account for one
		closed trail. The remaining $(n + 1) - 2 = n - 1$ right arrows
		in our path each correspond to a pair in a distinct closed
		trail. Hence we have found a total of $n$ closed trails, and so
		$p$ is closed.
	\item If we have two pairs of zeroes in our path, these account for two
		closed trails. The remaining $(n + 1) - 4 = n - 3$ right arrows
		in our path each correspond to a pair in a distinct closed
		trail. Hence, we have found $n - 1$ closed trails in our path.
		However, this means the remaining trail has to start and end in
		the same position, and so our path must be closed.
	\item If we have three pairs of zeroes in our path, these account for
		three closed trails. The remaining $(n + 1) - 6 = n - 5$ right
		arrows in our path correspond to distinct pairs in our trail.
		Hence, we have found $n - 2$ closed trails in our path. The
		remaining trails in our path use only left arrows. As the
		$\tau$-sequence in question has more than one zero, all other
		numbers in the sequence
		are less than $k$, hence each corresponding rule in our
		path contains both a right and a left rule, hence there are 6
		left arrows in these two trails, so we may apply
		Corollary~\ref{theyre_a_transposition} to show all trails
		in the path are closed.
\end{itemize}
\end{proof}

\begin{theorem}
For $n = 2k + 1$, the G\'{o}mez graphs $G_m \in \mywgp{\Pi_n}$ satisfy
$\myaut{G_m} \cong S_m$.
\end{theorem}

\begin{proof}
We have shown that the paths of length $n + 1$ which lead from a vertex back to
itself correspond to $\tau$-sequences. Hence, for the vertex
$e \in \myvert{\Gamma_n}$ the number of paths from each $\pi_i \in \Pi_n$ to $e$
of length $n$ is distinct for each $i$. This shows that the G\'{o}mez graphs are
subregular. Finally, we notice there is one path of length $n$ from $\pi_k$ to
$e$, but as $\pi_k$ is simply an $n-cycle$ we have $\pi_k^n = e$, hence there is
also a path of length $n - 1$ from $\pi_k$ to $e$. Hence, each $pi_i$ satisfies
either $\mydist{\pi_i}{e} < n$ or there is more than one path of length $n$ from
$\pi_i$ to $e$. Hence the G\'{o}mez graphs are alphabet stable.
\end{proof}

\section{The Even Case}

In this section, we deal with the case where $n = 2k$ and $k > 1$.
We will consider an
arbitrary closed path $p = p_1 p_2 \dots p_{n + 1}$.

\begin{lemma}
\label{even_zero_structure}
If $p_1 = \pi_0$, then $p_{k + 2} = \pi_1$, and the closed trail starting at
$k + 1$ contains two right arrows.
\end{lemma}

\begin{proof}
The closed trail starting at $k + 1$ is the trail starting with the right arrow
of $p_1$. Hence, if we assume there are no further right arrows in this trail,
we may apply Lemma~\ref{right_arrow_first} to show that $p_{n + 1}$ contains a
left arrow mapping $1$ to $k + 1$. However, no such right arrow exists.
Therefore we may use Lemma~\ref{right_arrows_rule} to show
to show the trail contains two right arrows.  The second right
arrow must have size $(n - 1) - (k - 1) = k$ so the number of spaces mapped
forwards equals those mapped backwards. Hence, the other right arrow is from
rule $\pi_1$, and this is in the trail after being mapped backwards $k + 1$
positions after $\pi_0$, hence occurs in position $p_{k + 2}$.
\end{proof}

\begin{corollary}
The rule $\pi_0$ occurs at most three times in $p$.
\end{corollary}

\begin{proof}
This is an immediate consequence of Lemma~\ref{even_zero_structure} and
Lemma~\ref{left_arrows_suck}.
\end{proof}

\begin{lemma}
If $p_1 = \pi_1$, then either $p_{k + 1} = \pi_0$ or $p_{n + 1} = \pi_0$.
\end{lemma}

\begin{proof}
Consider the trail starting with the right arrow of $p_1$.
If this trail contains another right arrow, then Lemma~\ref{right_arrows_rule}
shows it contains exactly two right arrows, and following previous logic the
other right arrow is $\pi_0$ at position $p_{k + 1}$. Otherwise we may apply
Lemma~\ref{right_left_pair} to get that $p_{n + 1} = \pi_0$.
\end{proof}

\begin{lemma}
\label{even_decreasing}
If $p_1 = \pi_i$ for some $i \geq 2$, then $p_{n + 1} = \pi_{i - 1}$.
\end{lemma}

\begin{proof}
Here we follow the same reasoning as Lemma~\ref{decreasing}.
\end{proof}

\begin{lemma}
\label{non_exist_len_n}
If $p$ is a closed path of length $n$,
no rule $\pi_i$ where $0 < i < k$ may occur in $p$.
\end{lemma}

\begin{proof}
Suppose $p$ contains some $\pi_i$ with
$0 < i < k$. Rotate $p$ as necessary so that $p_1 = \pi_i$. Consider the trail
starting with the right arrow of $\pi_i$. The right arrow of $\pi_i$ maps
$(k - 1) + i$ spaces forward. If this trail contains another right arrow, the
total distance mapped forward in the trail is at least $2(k - 1) + i > n - 2$.
The total distance mapped backward in the, however, trail is at most $n - 2$.
Therefore, this trail must contain a left arrow. This means that there is some
$j$ such that $(p_2 p_3 \dots p_j)(n) = 1$. The smallest $j$ this can occur for
is $j = n$, but this leaves no further rule to contain a left arrow in the
trail. Therefore, no such path can exist.
\end{proof}

\begin{proposition}
The path $p$ is closed if, and only if, $p_i = \pi_{a_i}$ for some
$\sigma$-sequence $\langle a_i \rangle$.
\end{proposition}

\begin{proof}
We follow the same approach as the proof of
Proposition~\ref{theyre_tau_sequences}.
\end{proof}

\begin{theorem}
For $n = 2k$, the G\'{o}mez graphs $G_m \in \mywgp{\Pi_n}$ satisfy $\myaut{G_m}
\cong S_m$.
\end{theorem}

\begin{proof}
We have shown that the paths of length $n + 1$ from each vertex in $\Gamma_n$
back to itself corresponds to a $\sigma$-sequence. Hence, for each
$e \in \Gamma_n$ the number of paths from each $\pi_i \in \Pi_n$ to $e$ of
length $n$ is distinct for each $i$, with the possible exception of $\pi_0$.
Further, we always have at least two paths of length $n$ from each $\pi_i$ to
$e$, hence we have alphabet stability. Now to show subregularity, we consider
paths of length $n - 1$ from each $\pi_i$ to $e$. Lemma~\ref{non_exist_len_n}
shows that only $\pi_0$ and $\pi_k$ may have a path of length $n - 1$ to $e$,
and $\pi_0$ is two disjoint $k$-cycles and hence $\pi_0^n = e$ gives us a path
of length $n - 1$ from $\pi_0$ to $e$. Therefore, if $\Gamma_n$ is not regular
there is an automorphism which exchanges $\pi_0$ and $\pi_k$. However, there are
only two paths of length $n$ from $\pi_k$ to $e$, and at least three paths of
length $n$ from $\pi_0$ to $e$. Hence we have established subregularity.
\end{proof}

\section{Problems in Other Cases}

The above argument may seem somewhat unsatisfying as it begins with the
observation of some reasonably general facts of G\'omez graphs but then only
goes on to make arguments of path counting in the cases $\mydg{k}{k}$ and
$\mydg{k}{k+1}$. However, we will now give an example to show that, though
perhaps the counting argument could be generalised to count all similar paths in
G\'omez graphs, this would not serve our purpose of classifying when G\'omez
graphs have full automorphism group $S_n$.

We include here a computed table showing numbers of closed paths of length
$k + 1$ starting with each $\pi_i$, written in the order $\pi_0, \pi_2$, etc.
Noting that the cases we have addressed are $k = k'$ and $k = k' + 1$.
\begin{table}[H]
\centering
\begin{tabular}{cr|c|c|c|c|c}
  & & \multicolumn{5}{c}{$k$} \\
  & & 2  & 3     & 4          & 5              & 6                  \\
  \hline \multirow{3}{*}{$k'$}&
 1 & 2,2 & 4,5,5 & 8,11,15,11 & 16,23,37,37,23 & 32,47,83,100,83,47 \\
&3 & -   & 1,2   & 2,5,3      & 4,12,12,12     & 8,27,35,44,33      \\
&5 & -   & -     & -          & 1,2,4          & 2,5,13,8           \\
\end{tabular}
\end{table}
\noindent
From this table we see that the first difficulty we encounter with the case
$k = k' + 2$ occurs for $k' = 3$, where there are twelve closed paths of length
$6$ starting from each of $\pi_1, \pi_2$ and $\pi_3$. This problem cannot be
resolved with the methods used to address the cases $k' = k$ and $k' = k + 1$.

In addition, this table highlights the interesting case of $k' = 1$. In this
case, we see the number of closed paths of length $k + 1$ starting with each
$\pi_i$ is equal to those starting from $\pi_{k - i}$ (taking $\pi_k = \pi_0$ in
the special case). We shall provide an informal proof of this to demonstrate
that this difficulty cannot possibly be overcome to make this style of proof
work for the case $k' = 1$.

We consider the case $k = 8$ and $k' = 1$. In this case, we have the rules
$\pi_0, \pi_1, \pi_2, \pi_3, \pi_4, \pi_5, \pi_6$ and $\pi_7$ as follows:
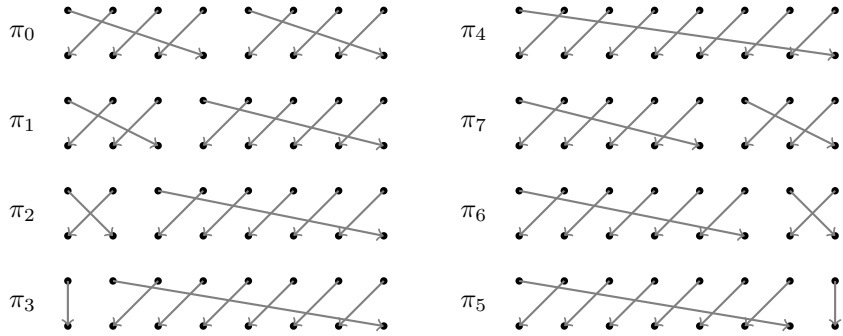
\begin{center}
\begin{tikzpicture}[scale=0.6]
	\foreach \x in {1,...,8} {
		\draw (\x, 0) node {};
		\draw (\x, 1) node {};
		\draw (\x, 2) node {};
		\draw (\x, 3) node {};
		\draw (\x, 4) node {};
		\draw (\x, 5) node {};
		\draw (\x, 6) node {};
		\draw (\x, 7) node {};
	}
	\foreach \x in {11,...,18} {
		\draw (\x, 0) node {};
		\draw (\x, 1) node {};
		\draw (\x, 2) node {};
		\draw (\x, 3) node {};
		\draw (\x, 4) node {};
		\draw (\x, 5) node {};
		\draw (\x, 6) node {};
		\draw (\x, 7) node {};
	}
	\draw [->, gray] (1, 7) to (4, 6);
	\draw [->, gray] (2, 7) to (1, 6);
	\draw [->, gray] (3, 7) to (2, 6);
	\draw [->, gray] (4, 7) to (3, 6);
	\draw [->, gray] (5, 7) to (8, 6);
	\draw [->, gray] (6, 7) to (5, 6);
	\draw [->, gray] (7, 7) to (6, 6);
	\draw [->, gray] (8, 7) to (7, 6);
	\node[fill=none, draw=none] at (0,  6.5) {$\pi_0$};
	\draw [->, gray] (1, 5) to (3, 4);
	\draw [->, gray] (2, 5) to (1, 4);
	\draw [->, gray] (3, 5) to (2, 4);
	\draw [->, gray] (4, 5) to (8, 4);
	\draw [->, gray] (5, 5) to (4, 4);
	\draw [->, gray] (6, 5) to (5, 4);
	\draw [->, gray] (7, 5) to (6, 4);
	\draw [->, gray] (8, 5) to (7, 4);
	\node[fill=none, draw=none] at (0,  4.5) {$\pi_1$};
	\draw [->, gray] (1, 3) to (2, 2);
	\draw [->, gray] (2, 3) to (1, 2);
	\draw [->, gray] (3, 3) to (8, 2);
	\draw [->, gray] (4, 3) to (3, 2);
	\draw [->, gray] (5, 3) to (4, 2);
	\draw [->, gray] (6, 3) to (5, 2);
	\draw [->, gray] (7, 3) to (6, 2);
	\draw [->, gray] (8, 3) to (7, 2);
	\node[fill=none, draw=none] at (0,  2.5) {$\pi_2$};
	\draw [->, gray] (1, 1) to (1, 0);
	\draw [->, gray] (2, 1) to (8, 0);
	\draw [->, gray] (3, 1) to (2, 0);
	\draw [->, gray] (4, 1) to (3, 0);
	\draw [->, gray] (5, 1) to (4, 0);
	\draw [->, gray] (6, 1) to (5, 0);
	\draw [->, gray] (7, 1) to (6, 0);
	\draw [->, gray] (8, 1) to (7, 0);
	\node[fill=none, draw=none] at (0,  0.5) {$\pi_3$};
	\draw [->, gray] (11, 7) to (18, 6);
	\draw [->, gray] (12, 7) to (11, 6);
	\draw [->, gray] (13, 7) to (12, 6);
	\draw [->, gray] (14, 7) to (13, 6);
	\draw [->, gray] (15, 7) to (14, 6);
	\draw [->, gray] (16, 7) to (15, 6);
	\draw [->, gray] (17, 7) to (16, 6);
	\draw [->, gray] (18, 7) to (17, 6);
	\node[fill=none, draw=none] at (10,  6.5) {$\pi_4$};
	\draw [->, gray] (11, 5) to (15, 4);
	\draw [->, gray] (12, 5) to (11, 4);
	\draw [->, gray] (13, 5) to (12, 4);
	\draw [->, gray] (14, 5) to (13, 4);
	\draw [->, gray] (15, 5) to (14, 4);
	\draw [->, gray] (16, 5) to (18, 4);
	\draw [->, gray] (17, 5) to (16, 4);
	\draw [->, gray] (18, 5) to (17, 4);
	\node[fill=none, draw=none] at (10,  4.5) {$\pi_7$};
	\draw [->, gray] (11, 3) to (16, 2);
	\draw [->, gray] (12, 3) to (11, 2);
	\draw [->, gray] (13, 3) to (12, 2);
	\draw [->, gray] (14, 3) to (13, 2);
	\draw [->, gray] (15, 3) to (14, 2);
	\draw [->, gray] (16, 3) to (15, 2);
	\draw [->, gray] (17, 3) to (18, 2);
	\draw [->, gray] (18, 3) to (17, 2);
	\node[fill=none, draw=none] at (10,  2.5) {$\pi_6$};
	\draw [->, gray] (11, 1) to (17, 0);
	\draw [->, gray] (12, 1) to (11, 0);
	\draw [->, gray] (13, 1) to (12, 0);
	\draw [->, gray] (14, 1) to (13, 0);
	\draw [->, gray] (15, 1) to (14, 0);
	\draw [->, gray] (16, 1) to (15, 0);
	\draw [->, gray] (17, 1) to (16, 0);
	\draw [->, gray] (18, 1) to (18, 0);
	\node[fill=none, draw=none] at (10,  0.5) {$\pi_5$};
\end{tikzpicture}
\end{center}
\noindent
We consider a diagram of the closed path
$\pi_2 \pi_3 \pi_7 \pi_7 \pi_0 \pi_1 \pi_2 \pi_3 \pi_2$ and note that if we
rotate it by 180 degrees and reverse the arrows we get another closed path.
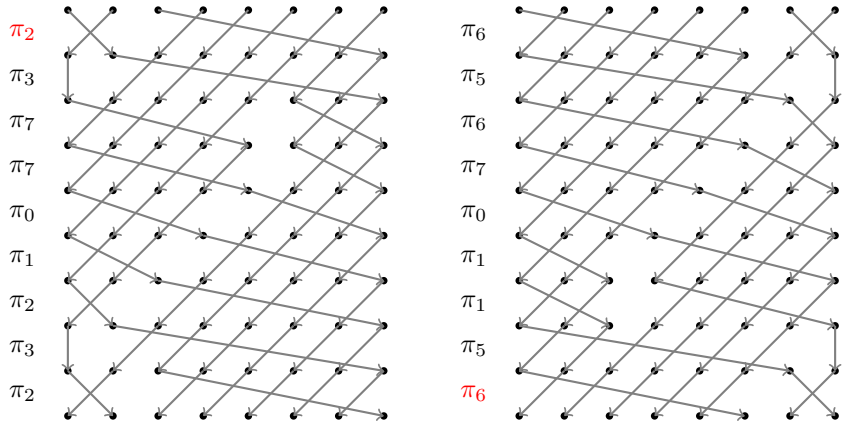
\begin{center}
\begin{tikzpicture}[scale=0.6]
	\foreach \x in {1,...,8} {
		\draw (\x, 0) node {};
		\draw (\x, 1) node {};
		\draw (\x, 2) node {};
		\draw (\x, 3) node {};
		\draw (\x, 4) node {};
		\draw (\x, 5) node {};
		\draw (\x, 6) node {};
		\draw (\x, 7) node {};
		\draw (\x, 8) node {};
		\draw (\x, 9) node {};
	}
	\foreach \x in {11,...,18} {
		\draw (\x, 0) node {};
		\draw (\x, 1) node {};
		\draw (\x, 2) node {};
		\draw (\x, 3) node {};
		\draw (\x, 4) node {};
		\draw (\x, 5) node {};
		\draw (\x, 6) node {};
		\draw (\x, 7) node {};
		\draw (\x, 8) node {};
		\draw (\x, 9) node {};
	}
	\draw [->, gray] (1, 9) to (2, 8);
	\draw [->, gray] (2, 9) to (1, 8);
	\draw [->, gray] (3, 9) to (8, 8);
	\draw [->, gray] (4, 9) to (3, 8);
	\draw [->, gray] (5, 9) to (4, 8);
	\draw [->, gray] (6, 9) to (5, 8);
	\draw [->, gray] (7, 9) to (6, 8);
	\draw [->, gray] (8, 9) to (7, 8);
	\node[fill=none, draw=none] at (0,  8.5) {\textcolor{red}{$\pi_2$}};
	\draw [->, gray] (1, 8) to (1, 7);
	\draw [->, gray] (2, 8) to (8, 7);
	\draw [->, gray] (3, 8) to (2, 7);
	\draw [->, gray] (4, 8) to (3, 7);
	\draw [->, gray] (5, 8) to (4, 7);
	\draw [->, gray] (6, 8) to (5, 7);
	\draw [->, gray] (7, 8) to (6, 7);
	\draw [->, gray] (8, 8) to (7, 7);
	\node[fill=none, draw=none] at (0,  7.5) {$\pi_3$};
	\draw [->, gray] (1, 7) to (5, 6);
	\draw [->, gray] (2, 7) to (1, 6);
	\draw [->, gray] (3, 7) to (2, 6);
	\draw [->, gray] (4, 7) to (3, 6);
	\draw [->, gray] (5, 7) to (4, 6);
	\draw [->, gray] (6, 7) to (8, 6);
	\draw [->, gray] (7, 7) to (6, 6);
	\draw [->, gray] (8, 7) to (7, 6);
	\node[fill=none, draw=none] at (0,  6.5) {$\pi_7$};
	\draw [->, gray] (1, 6) to (5, 5);
	\draw [->, gray] (2, 6) to (1, 5);
	\draw [->, gray] (3, 6) to (2, 5);
	\draw [->, gray] (4, 6) to (3, 5);
	\draw [->, gray] (5, 6) to (4, 5);
	\draw [->, gray] (6, 6) to (8, 5);
	\draw [->, gray] (7, 6) to (6, 5);
	\draw [->, gray] (8, 6) to (7, 5);
	\node[fill=none, draw=none] at (0,  5.5) {$\pi_7$};
	\draw [->, gray] (1, 5) to (4, 4);
	\draw [->, gray] (2, 5) to (1, 4);
	\draw [->, gray] (3, 5) to (2, 4);
	\draw [->, gray] (4, 5) to (3, 4);
	\draw [->, gray] (5, 5) to (8, 4);
	\draw [->, gray] (6, 5) to (5, 4);
	\draw [->, gray] (7, 5) to (6, 4);
	\draw [->, gray] (8, 5) to (7, 4);
	\node[fill=none, draw=none] at (0,  4.5) {$\pi_0$};
	\draw [->, gray] (1, 4) to (3, 3);
	\draw [->, gray] (2, 4) to (1, 3);
	\draw [->, gray] (3, 4) to (2, 3);
	\draw [->, gray] (4, 4) to (8, 3);
	\draw [->, gray] (5, 4) to (4, 3);
	\draw [->, gray] (6, 4) to (5, 3);
	\draw [->, gray] (7, 4) to (6, 3);
	\draw [->, gray] (8, 4) to (7, 3);
	\node[fill=none, draw=none] at (0,  3.5) {$\pi_1$};
	\draw [->, gray] (1, 3) to (2, 2);
	\draw [->, gray] (2, 3) to (1, 2);
	\draw [->, gray] (3, 3) to (8, 2);
	\draw [->, gray] (4, 3) to (3, 2);
	\draw [->, gray] (5, 3) to (4, 2);
	\draw [->, gray] (6, 3) to (5, 2);
	\draw [->, gray] (7, 3) to (6, 2);
	\draw [->, gray] (8, 3) to (7, 2);
	\node[fill=none, draw=none] at (0,  2.5) {$\pi_2$};
	\draw [->, gray] (1, 2) to (1, 1);
	\draw [->, gray] (2, 2) to (8, 1);
	\draw [->, gray] (3, 2) to (2, 1);
	\draw [->, gray] (4, 2) to (3, 1);
	\draw [->, gray] (5, 2) to (4, 1);
	\draw [->, gray] (6, 2) to (5, 1);
	\draw [->, gray] (7, 2) to (6, 1);
	\draw [->, gray] (8, 2) to (7, 1);
	\node[fill=none, draw=none] at (0,  1.5) {$\pi_3$};
	\draw [->, gray] (1, 1) to (2, 0);
	\draw [->, gray] (2, 1) to (1, 0);
	\draw [->, gray] (3, 1) to (8, 0);
	\draw [->, gray] (4, 1) to (3, 0);
	\draw [->, gray] (5, 1) to (4, 0);
	\draw [->, gray] (6, 1) to (5, 0);
	\draw [->, gray] (7, 1) to (6, 0);
	\draw [->, gray] (8, 1) to (7, 0);
	\node[fill=none, draw=none] at (0,  0.5) {$\pi_2$};
	\draw [->, gray] (11, 9) to (16, 8);
	\draw [->, gray] (12, 9) to (11, 8);
	\draw [->, gray] (13, 9) to (12, 8);
	\draw [->, gray] (14, 9) to (13, 8);
	\draw [->, gray] (15, 9) to (14, 8);
	\draw [->, gray] (16, 9) to (15, 8);
	\draw [->, gray] (17, 9) to (18, 8);
	\draw [->, gray] (18, 9) to (17, 8);
	\node[fill=none, draw=none] at (10,  8.5) {$\pi_6$};
	\draw [->, gray] (11, 8) to (17, 7);
	\draw [->, gray] (12, 8) to (11, 7);
	\draw [->, gray] (13, 8) to (12, 7);
	\draw [->, gray] (14, 8) to (13, 7);
	\draw [->, gray] (15, 8) to (14, 7);
	\draw [->, gray] (16, 8) to (15, 7);
	\draw [->, gray] (17, 8) to (16, 7);
	\draw [->, gray] (18, 8) to (18, 7);
	\node[fill=none, draw=none] at (10,  7.5) {$\pi_5$};
	\draw [->, gray] (11, 7) to (16, 6);
	\draw [->, gray] (12, 7) to (11, 6);
	\draw [->, gray] (13, 7) to (12, 6);
	\draw [->, gray] (14, 7) to (13, 6);
	\draw [->, gray] (15, 7) to (14, 6);
	\draw [->, gray] (16, 7) to (15, 6);
	\draw [->, gray] (17, 7) to (18, 6);
	\draw [->, gray] (18, 7) to (17, 6);
	\node[fill=none, draw=none] at (10,  6.5) {$\pi_6$};
	\draw [->, gray] (11, 6) to (15, 5);
	\draw [->, gray] (12, 6) to (11, 5);
	\draw [->, gray] (13, 6) to (12, 5);
	\draw [->, gray] (14, 6) to (13, 5);
	\draw [->, gray] (15, 6) to (14, 5);
	\draw [->, gray] (16, 6) to (18, 5);
	\draw [->, gray] (17, 6) to (16, 5);
	\draw [->, gray] (18, 6) to (17, 5);
	\node[fill=none, draw=none] at (10,  5.5) {$\pi_7$};
	\draw [->, gray] (11, 5) to (14, 4);
	\draw [->, gray] (12, 5) to (11, 4);
	\draw [->, gray] (13, 5) to (12, 4);
	\draw [->, gray] (14, 5) to (13, 4);
	\draw [->, gray] (15, 5) to (18, 4);
	\draw [->, gray] (16, 5) to (15, 4);
	\draw [->, gray] (17, 5) to (16, 4);
	\draw [->, gray] (18, 5) to (17, 4);
	\node[fill=none, draw=none] at (10,  4.5) {$\pi_0$};
	\draw [->, gray] (11, 4) to (13, 3);
	\draw [->, gray] (12, 4) to (11, 3);
	\draw [->, gray] (13, 4) to (12, 3);
	\draw [->, gray] (14, 4) to (18, 3);
	\draw [->, gray] (15, 4) to (14, 3);
	\draw [->, gray] (16, 4) to (15, 3);
	\draw [->, gray] (17, 4) to (16, 3);
	\draw [->, gray] (18, 4) to (17, 3);
	\node[fill=none, draw=none] at (10,  3.5) {$\pi_1$};
	\draw [->, gray] (11, 3) to (13, 2);
	\draw [->, gray] (12, 3) to (11, 2);
	\draw [->, gray] (13, 3) to (12, 2);
	\draw [->, gray] (14, 3) to (18, 2);
	\draw [->, gray] (15, 3) to (14, 2);
	\draw [->, gray] (16, 3) to (15, 2);
	\draw [->, gray] (17, 3) to (16, 2);
	\draw [->, gray] (18, 3) to (17, 2);
	\node[fill=none, draw=none] at (10,  2.5) {$\pi_1$};
	\draw [->, gray] (11, 2) to (17, 1);
	\draw [->, gray] (12, 2) to (11, 1);
	\draw [->, gray] (13, 2) to (12, 1);
	\draw [->, gray] (14, 2) to (13, 1);
	\draw [->, gray] (15, 2) to (14, 1);
	\draw [->, gray] (16, 2) to (15, 1);
	\draw [->, gray] (17, 2) to (16, 1);
	\draw [->, gray] (18, 2) to (18, 1);
	\node[fill=none, draw=none] at (10,  1.5) {$\pi_5$};
	\draw [->, gray] (11, 1) to (16, 0);
	\draw [->, gray] (12, 1) to (11, 0);
	\draw [->, gray] (13, 1) to (12, 0);
	\draw [->, gray] (14, 1) to (13, 0);
	\draw [->, gray] (15, 1) to (14, 0);
	\draw [->, gray] (16, 1) to (15, 0);
	\draw [->, gray] (17, 1) to (18, 0);
	\draw [->, gray] (18, 1) to (17, 0);
	\node[fill=none, draw=none] at (10,  0.5) {\textcolor{red}{$\pi_6$}};
\end{tikzpicture}
\end{center}
\noindent
Hence from the closed path
\[ p \triangleq \pi_2 \pi_3 \pi_7 \pi_7 \pi_0 \pi_1 \pi_2 \pi_3 \pi_2 \]
we form the closed path
\[ q \triangleq \pi_6 \pi_5 \pi_6 \pi_7 \pi_0 \pi_1 \pi_1 \pi_5 \pi_6. \]
As we have highlighted
in the above example, closed paths beginning with $\pi_2$ are in bijective
correspondence with closed paths ending in $\pi_6$. Finally, if we consider the
rotation of $q$ such that places the highlighted $\pi_6$ at the beginning, we
have found a bijective correspondence between closed paths of a given length
beginning with $\pi_2$ and those beginning with $\pi_6$. Hence, we cannot
consider any length of path to differentiate these two rules under automorphisms
of $\Gamma$.

\section{Conclusion}

In the context of the degree diameter problem in the directed case, the
possibility of finding larger graphs for given degree and diameter than the
G\'omez graphs remains open (and, indeed, it appears highly likely that larger
examples do exist). Hence the optimality result for the G\'omez graphs primarily
serves to demonstrate the limitations of this particular method of construction,
and that the construction of larger graphs will likely require altogether new
ideas.

Further, whilst we have shown an optimality result for the G\'{o}mez graphs
we have not shown that they are unique with this property.
The fact the cycles used in the G\'{o}mez graphs had the smaller
cycle on the left and larger cycle on the right is not necessarily required.
Therefore, an interesting question would be to determine which of the potential
optimal constructions work. Beyond this, if other constructions work, it is
possible that they could have larger automorphism groups.

In order to raise further questions, we now define \textit{G\'omez like}
graphs. Suppose a set of permutations $\Pi \subseteq S_n$ contains at least one
permutation containing a cycle of each length up to $n$, and $\Pi$ is as small
as possible with this property (i.e. $\Pi$ meets our optimality condition from
previously). If the word graph family $\mywgp{\Pi}$ is diameter $n$, then we
shall call these graphs G\'omez like. An obvious first question regarding
G\'omez like graphs is what conditions such a set $\Pi$ needs to fulfil to be
admissible.

With regards to the classification of the automorphism groups of G\'omez graphs,
what classification has been achieved misses out a number of important cases. In
rough order of importance, these are
\begin{enumerate}[label=\textnormal{(\roman*)}]
	\item the automorphism groups of undirected G\'omez graphs,
	\item the automorphism groups of the graphs $\mydg{k}{k'}$ for
		$k \geq k' + 2$,
	\item the automorphism groups of $\mydg{k}{1}$,
	\item the automorphism groups of G\'omez like graphs.
\end{enumerate}

A particular question considered by the author was whether a set of permutations
$\Pi$ is admissible for shift restricted word graphs if, and only if, the
corresponding alphabet fixing subgraph $\Gamma$ is $n$-reachable. The reason for
this question is that both the Faber-Moore-Chen graphs and the G\'omez graphs
have this property, and further a simple argument shows that admissibility
implies a weaker but similar property as we now show.

\begin{lemma}
If $\Pi \subseteq S_n$ is an admissible set of permutations, letting $k < n$ and
$m = n - k$, then for any $\tau \in S_n$ with
\[
	\tau(i) = \begin{cases}
		i - m, & \text{if $m < i \leq n$}, \\
		j,     & \text{otherwise},            \\
	\end{cases}
\]
there are some $\pi_1, \pi_2, \dots, \pi_m \in \Pi$ such that
$\pi_1 \pi_2 \dots \pi_m = \tau$.
\end{lemma}

\begin{proof}
Similar to Lemma~\ref{contains_cycle}. We simply consider a path from
$x_1 x_2 \dots x_n$ to
$y_1 y_2 \dots y_k x_{\tau(k + 1)} x_{\tau(k + 2)} \dots x_{\tau(n)}$.
Consider the case $n = 9, k = 3, m = 6$ here
\[
	\newcommand{\f}[1]{\textcolor{red}{#1}}
	\newcommand{\g}[1]{\textcolor{blue}{#1}}
	\begin{matrix}
	&  x_1 &   x_2 &   x_3 &\f{x_4}&\f{x_5}&\f{x_6}&\f{x_7}&\f{x_8}&\f{x_9}\\
	&  x_2 &   x_3 &\f{x_4}&\f{x_5}&\f{x_6}&\f{x_7}&\f{x_8}&\f{x_9}&\g{y_1}\\
	&  x_3 &\f{x_4}&\f{x_5}&\f{x_6}&\f{x_7}&\f{x_8}&\f{x_9}&\g{y_1}&\g{y_2}\\
	p_1&\f{x_4}&\f{x_5}&\f{x_6}&\f{x_7}&\f{x_8}&\f{x_9}&\g{y_1}&\g{y_2}&\g{y_3}\\
	&  x   &   x   &   x   &   x   &   x   &\g{y_1}&\g{y_2}&\g{y_3}&   x   \\
	&  x   &   x   &   x   &   x   &\g{y_1}&\g{y_2}&\g{y_3}&   x   &   x   \\
	&  x   &   x   &   x   &\g{y_1}&\g{y_2}&\g{y_3}&   x   &   x   &   x   \\
	&  x   &   x   &\g{y_1}&\g{y_2}&\g{y_3}&   x   &   x   &   x   &   x   \\
	&  x   &\g{y_1}&\g{y_2}&\g{y_3}&   x   &   x   &   x   &   x   &   x   \\
	p_2&\g{y_1}&\g{y_2}&\g{y_3}&   x   &   x   &   x   &   x   &   x   &   x   \\
	\end{matrix}
\]
Considering the path from $p_1$ to $p_2$ we find each permutation in this path
must be in $\Pi$ and the permutation of $x_4, x_5, \dots, x_9$ is arbitrary.
\end{proof}

\begin{corollary}
If $\tau \in S_n$ such that there exists some $k < n$ with
\[
	\tau(i) = \begin{cases}
		i' < k    & \text{if $i < k$}     \\
		i' \geq k & \text{if $i \geq k$},
	\end{cases}
\]
then there exist $\pi_1, \pi_2, \dots, \pi_n \in \Pi$ such that
$\tau = \pi_1 \pi_2 \dots \pi_n$.
\end{corollary}

\begin{proof}
This is simply the concatenation of two of the previous permutations we showed
exist in the previous lemma.
\end{proof}

Hence, if $\Pi$ is admissible, we can easily see ``a lot'' of permutations must
be $n$-reachable. This taken in conjunction with the fact the known optimal
admissible $\Pi$ are $n$-reachable suggests that this may be a necessary
requirement.

\bibliographystyle{plain}
\bibliography{paper}

\end{document}